\documentclass[11pt,a4paper]{amsart}
\usepackage{amssymb,amsmath,amsthm}
\usepackage[pdftex]{graphicx}

\usepackage[colorlinks=true, pdfstartview=FitV, linkcolor=blue, citecolor=blue, urlcolor=blue,pagebackref=false]{hyperref}

\usepackage{esint}
\usepackage{mathrsfs}

\usepackage{mathtools}

\usepackage{times,latexsym,color}

\newcommand{\BOX}{\ensuremath\Box}

\newtheorem{theorem}{Theorem}
\newtheorem*{theorem*}{Theorem}
\newtheorem{pro}{Proposition}
\newtheorem{lemma}{Lemma}
\newtheorem{cor}{Corollary}

\theoremstyle{remark}
\newtheorem{remark}{Remark}

\theoremstyle{definition}
\newtheorem{definition}{Definition}

\newtheorem*{weakstrong}{  Theorem 1.3  [3] }

\newcommand{\N}{\mathbb{N}}

\newcommand{\R}{\mathbb{R}}
\newcommand{\Z}{\mathbb{Z}}

\definecolor{darkgreen}{rgb}{0,0.5,0}
\definecolor{darkblue}{rgb}{0,0,0.7}
\definecolor{darkred}{rgb}{0.9,0.1,0.1}
\definecolor{lightblue}{rgb}{0,0.51,1}

\DeclarePairedDelimiter\ceil{\lceil}{\rceil}
\DeclarePairedDelimiter\floor{\lfloor}{\rfloor}

\begin{document}

\title{Local Hadamard well-posedness results for the Navier-Stokes equations}

\author[T. Barker]{Tobias Barker\\\\\textbf{\,\,\,\,\,\,\,\,\,\,\,\,\,\,\,\,\,\,\,\,\,\,\,\,\,\,\,\,\,\,\,\,\,\,\,\,\,\,\,\,\,\,\,\,\,\,\,In memory of my stepfather Brian Ruddle (1950-2019)}}
\address[T. Barker]{DMA, \'{E}cole Normale Sup\'erieure, CNRS, PSL Research University, 75\, 005 Paris}
\email{tobiasbarker5@gmail.com}

\subjclass[2010]{}
\date{\today}

\maketitle

\noindent {\bf Abstract} 
In this paper we consider classes of initial data that ensure local-in-time Hadamard well-posedness of the associated weak Leray-Hopf solutions of the three-dimensional Navier-Stokes equations.  In particular, for any solenodial $L_{2}$ initial data $u_{0}$ belonging to certain subsets of $VMO^{-1}(\mathbb{R}^3)$, we show that weak Leray-Hopf solutions depend continuously with respect to small divergence-free $L_{2}$ perturbations of the initial data $u_{0}$ (on some finite-time interval). Our main result is inspired and improves upon previous work of the author \cite{barker2018} and work of Jean-Yves Chemin \cite{chemin}. Our method builds upon \cite{barker2018} and \cite{chemin}. In particular our method hinges on decomposition results for the initial data inspired by Calder\'{o}n \cite{Calderon90} together with use of persistence of regularity results. The persistence of regularity statement presented may be of independent interest, since it does not rely upon the solution or the initial data being in the perturbative regime.

\vspace{0.3cm}

\vspace{0.3cm}


\section{Introduction}
At the beginning of the 20th century, Jacques Hadamard introduced a notion of well-posedness of partial differential equations. In particular, a evolutionary partial differential equation is said to be \textit{Hadamard well-posed} if
\begin{enumerate}
\item \textbf{(Existence)} A solution exists for all time.
\item \textbf{(Uniqueness)} The solution is unique for all time.
\item \textbf{(Continuous Dependence)} The solution depends continuously on the initial data.
\end{enumerate} 
The issue of Hadamard well-posedness depends not only on the equation under consideration, but also on the notion of `solution' and the classes  considered for the initial data.

For the Navier-Stokes equations, a popular notion of solution (with certain physical relevance) is that of \textit{weak Leray-Hopf solutions}. In particular, for any $L_{2}(\mathbb{R}^3)$ divergence-free initial data we say that $u:\mathbb{R}^3\times \rightarrow \mathbb{R}^3$ is a weak Leray-Hopf solution associated to $u_{0}$ if
\begin{itemize}
\item $u\in C_{w}([0,\infty); J(\mathbb{R}^3)\cap L^{2}(0,\infty; \dot{H}^{1}(\mathbb{R}^3))$\footnote{  Throughout this paper $J(\mathbb{R}^3):=\{u_{0}\in L_{2}(\mathbb{R}^3)$: $\textrm{div}\,u_{0}=0\}.$ $C_{w}([0,\infty); J(\mathbb{R}^3))$ denotes continuity in time with respect to the weak $L_{2}$ topology.}.
\item  $u$ solves the Navier-stokes equations in the distributional sense:
$$ \partial_{t}u-\Delta u+ u\cdot\nabla u+\nabla p=0\,\,\,\,\textrm{in}\,\,\,\,\mathbb{R}^3\times (0,\infty),\,\,\,\,\textrm{div}\,u=0,\,\,\,\,u(\cdot,0)=u_{0}.$$
\item $u$ satisfies the energy inequality for all $t\geq 0$:
$$\|u(\cdot,t)\|_{L_{2}(\mathbb{R}^3)}^2+2\int\limits_{0}^{t}\int\limits_{\mathbb{R}^3} |\nabla u(y,s)|^2 dyds\leq \|u_{0}\|_{L_{2}(\mathbb{R}^3)}^2.  $$
\end{itemize}
For any $u_{0}\in J(\mathbb{R}^3)$, global-in-time existence of an associated weak Leray-Hopf solution of the Navier-Stokes equations was established by Leray in \cite{Le} in 1934. Up to the present date, whether or not weak Leray-Hopf solutions are unique remains an outstanding open problem in mathematical fluid mechanics. Recently, sufficient conditions for nonuniqueness were provided in \cite{jiasverak2015} and numerical evidence that these sufficient conditions hold was provided in \cite{guillodsverak}.

Let us now give a definition that expresses the continuous dependence requirement for Hadamard well-posedness in the context of weak Leray-Hopf solutions.
\begin{definition}\label{continuousdependence}
Let $u_{0}\in L_{2}(\mathbb{R}^3)$ be weakly divergence-free. We say that weak Leray-Hopf solutions are `\textbf{\textit{locally continuously dependent with respect to $u_{0}$}}' if the following holds true.\\
There exists a finite positive $T$, $\varepsilon>0$ and a continuous function $\Psi$ with $\Psi(0)=0$ such that if
\begin{itemize}
\item $v_{0}\in B_{L_{2}}(u_{0},\varepsilon):=\{w_{0}\in J(\mathbb{R}^3): \|w_{0}-u_{0}\|_{L_{2}}<\varepsilon\}$
\item $v(\cdot,v_{0})$ and $u(\cdot,u_{0})$ are global-in-time weak Leray-Hopf solutions associated to $u_{0}$ and $v_{0}$
\end{itemize}
then for all $t\in (0,T]$ one has the estimate
\begin{equation}\label{ctsdependentest}
\|v(\cdot,t)-u(\cdot,t)\|_{L_{2}(\mathbb{R}^3)}^2+ 2\int\limits_{0}^{t}\int\limits_{\mathbb{R}^3} |\nabla (v-u)|^2 dyds\leq \Psi(\|v_{0}-u_{0}\|_{L_{2}(\mathbb{R}^3)}).
\end{equation}
\end{definition}
Whilst the above definition expresses the notion of Hadamard's continuous dependence condition in the context of weak Leray-Hopf solutions, it also has  ramifications for the regularity of solutions with initial data close to those which generate smooth solutions. In particular, suppose that
\begin{itemize}
\item[] a) Weak Leray-Hopf solutions are locally continuously dependent with respect to $u_{0}$.
\item[] b) The  weak Leray-Hopf solution $u(\cdot,u_{0})$ (unique on $(0,T)$) belongs to $C^{\infty}(\mathbb{R}^3\times (0,T])$.
\end{itemize}
Then a)-b) imply that for any compact set $K$ contained in $\mathbb{R}^3\times (0,T)$, there exists $\varepsilon(K,\Psi)$ such that
$$\|v_{0}-u_{0}\|_{L_{2}(\mathbb{R}^3)}<\varepsilon(K,\Psi)\Rightarrow \,\,\,\,\textrm{any}\,\,\,\,\textrm{suitable}\footnote{We say that a weak Leray-Hopf solution $(v,q)$ is suitable on $\mathbb{R}^3\times (0,T)$ if $$2\int\limits_{0}^{T}\int\limits_{\mathbb{R}^3} |\nabla v|^2\varphi dxdt\leq 
\int\limits_{0}^{T}\int\limits_{\mathbb{R}^3}[ |v|^2(\partial_{t}\varphi+\Delta\varphi)+(|v|^2+2q)v\cdot\varphi]dxdt  $$ for every non-negative $\varphi\in C_{0}^{\infty}(\mathbb{R}^3\times (0,T))$}\,\textrm{weak}\,\textrm{Leray-Hopf}\,\textrm{solution}\, v(\cdot, v_{0})\in L^{\infty}_{x,t}(K). $$
Such an statements follow immediately from a contradiction argument and the `\textit{persistence of singularities'} in \cite{rusinsver}.\\
In this paper, we are concerned with the following natural question:
\begin{itemize}
\item[] \textbf{(Q) Which $\mathcal{Z}\subset\mathcal{S}^{'}(\mathbb{R}^3)$ are such that $u_{0}\in J(\mathbb{R}^3)\cap\mathcal{Z}$ implies that weak Leray-Hopf solutions are locally continuously dependent with respect to $u_{0}$?}
\end{itemize}
From Definition \ref{continuousdependence}, we see that positive answers to \textbf{(Q)} provide classes of initial data for which weak Leray-Hopf solutions are Hadamard well-posed locally in time.

In \cite{barker2018}, the author provided the current widest\footnote{For the subclass of weak Leray-Hopf solutions called `\textit{local Leray solutions}', Lemarie Rieusset built upon ideas in \cite{barker2018} to show in \cite{LR19} that short-time uniqueness holds for a wider class of initial data than those considered in \cite{barker2018}.}class of initial data for which the associated weak Leray-Hopf solutions are unique on some time interval. In particular the following Theorem was proven in \cite{barker2018}.
\begin{weakstrong}\label{weakstrong1}
Suppose that there exists $q>3$ and $s\in(-1+\frac{2}{q},0)$ such that
\begin{equation}\label{initialdataassumptionmaintheoweakstrong}
u_{0}\in J(\mathbb{R}^3)\cap VMO^{-1}(\mathbb{R}^3)\cap\dot{{B}}^{s}_{q,q}(\mathbb{R}^3).
\end{equation} 
 Then, there exists a $\hat{T}(u_0)>0$ such that all weak Leray-Hopf solutions on $Q_{\infty}$, with initial data $u_0$, coincide  on $Q_{\hat{T}(u_0)}:=\mathbb{R}^3\times (0,\hat{T}(u_0)).$
\end{weakstrong} 
 The main result of this paper, which we state below, shows that for such classes of initial data, weak Leray-Hopf solutions are Hadamard well posed locally in time.

\begin{theorem}\label{maintheo}
Suppose that there exists $q>3$ and $s\in(-1+\frac{2}{q},0)$ such that
\begin{equation}\label{initialdataassumptionmaintheo}
u_{0}\in J(\mathbb{R}^3)\cap VMO^{-1}(\mathbb{R}^3)\cap\dot{{B}}^{s}_{q,q}.
\end{equation} 
Let $\hat{T}(u_{0})$ be as in the above Theorem and let $u$ be the unique Leray-Hopf solution associated with $u_{0}$. Then for any positive $\eta\in(0,1)$ there exists  $T(\eta, u_{0},s,q)\in (0, \hat{T}]$ and $C(\eta, u_{0},s,q)>0$ such that the following holds. For any weak Leray solution $v$ associated with $v_{0}$ with 
\begin{equation}\label{unitballu0}
\|v_{0}-u_{0}\|_{L^{2}(\mathbb{R}^3)}<1,
\end{equation}
 we have that for all $t\in [0, T]$
\begin{equation}\label{L2stabmaintheo}
\|v(t)-u(t)\|_{L_{2}}^2+\int\limits_{0}^{t} \|\nabla(v-u)(t')\|_{L_{2}}^2 dt'\leq C\|v_{0}-u_{0}\|_{L^{2}}^{2-2\eta}.
\end{equation}

\end{theorem}
\subsection{Comparison with previous literature}
The classical approach to determining $\mathcal{Z}$ such that (\textbf{Q}) holds true dates back to Leray in \cite{Le} (we refer to this as `Leray's approach'). Let us now describe this in more detail.

Let $v(\cdot,v_{0})$ and $u(\cdot,u_{0})$ be two weak Leray-Hopf solutions with $u_{0}\in \mathcal{Z}\cap J(\mathbb{R}^3)$, $v_{0}\in J(\mathbb{R}^3)$\footnote{Throughout this paper $J(\mathbb{R}^3):=\{u_{0}\in L_{2}(\mathbb{R}^3)$: $\textrm{div}\,u_{0}=0\}.$} and $w\equiv u-v$. In Leray's approach, one requires the existence of $u(\cdot,u_{0})$ in path spaces $\mathcal{X}_{T}$ possessing certain properties. In particular,
 $a\in \mathcal{X}_{T}$ and $b\,c\in C_{w}([0,T]; J(\mathbb{R}^3)\cap L^{2}([0,T]; \dot{H}^{1}(\mathbb{R}^3))$ $\Rightarrow$
$$F(a,b,c,t):=\int\limits_{0}^{t}\int\limits_{\mathbb{R}^3} (a\otimes b):\nabla c dxdt'<\infty$$
for $t\in [0,T]$ and satisfies certain continuity estimates (see, for example, \cite{dubois}). Once $\mathcal{X}_{T}$ satisfies this requirement, the approach in \cite{Le} gives a positive answer to \textbf{(Q)} by applying Gronwall's lemma to the energy inequality
\begin{equation}\label{energyestintroI}
\|w(\cdot,t)\|_{L^{2}_{x}}^2+2\int\limits_{0}^{t}\int\limits_{\mathbb{R}^3} |\nabla w|^2 dxdt'\leq \|u_{0}-v_{0}\|_{L^{2}_{x}}^2+ 2\int\limits_{0}^{t}\int\limits_{\mathbb{R}^3} u\otimes w: \nabla w dxdt'.
\end{equation}
Leray's approach was first used in \cite{Le} to show that for $\mathcal{Z}=H^{1}(\mathbb{R}^3)$ and $\mathcal{Z}= L_{p}(\mathbb{R}^3)$ ($3<p\leq\infty$) we have local-in-time Hadamard well-posedness of  of `turbulent solutions' (which are a subclass of Leray-Hopf solutions). Leray's approach has been applied to many other cases and we only attempt to list the cases most relevant to this paper. At the start of the $21^{st}$ century, \cite{GP2000} utilized Littlewood-Paley theory and Leray's approach to provide a positive answer for question \textbf{(Q)} for the homogeneous Besov spaces $$\mathcal{Z}= \dot{B}_{p,q}^{-1+\frac{3}{p}}(\mathbb{R}^3)$$ with $p, q<\infty$ and $$\frac{3}{p}+\frac{2}{q}\geq 1.$$ Certain further extensions were provided in \cite{dubois}. 

For the wider yet class\footnote{We denote $\dot{B}^{-1+\frac{3}{p}}_{p,\infty}$ to be the homogeneous Besov space and $\dot{\mathbb{B}}^{-1+\frac{3}{p}}_{p,\infty}$ to be the subspace of closure of  Schwartz functions.} $\mathcal{Z}=\dot{\mathbb{B}}^{-1+\frac{3}{p}}_{p,\infty}$ ($p\in (3,\infty)$), arguments in \cite{cannone}\footnote{For an exposition of these arguments, we also refer to \cite{barker2018}.} imply that there exists a $T(u_{0})$ and a weak Leray-Hopf solution $u(\cdot,u_{0})$ that is infinitely smooth on $\mathbb{R}^3\times (0,T(u_{0}))$. However, for this case the main difficulty is that it is unknown if Leray's approach is applicable. Specifically,  when $u_{0}$ belongs to the above class and $w$ belongs to the energy space (without assuming $w$ solves an equation) it is not known that this trilinear term $$\int\limits_{0}^{T}\int\limits_{\mathbb{R}^3} u(\cdot,u_{0})\otimes w:\nabla w dxdt' $$ in \eqref{energyestintroI} is even convergent.

 These difficulties were tackled by Jean-Yves Chemin in \cite{chemin}, which provided a positive answer to \textbf{(Q)} for $\mathcal{Z}=\dot{H}^{\alpha}(\mathbb{R}^3)\cap\dot{\mathbb{B}}^{-1+\frac{3}{p}}_{p,\infty}$ ($\alpha>0$ and $p\in (3,\infty)$) by means of the following theorem.
\begin{theorem}\label{Chemin}
Suppose that there exists $\alpha>0$ and $p\in(1,\infty)$ such that
\begin{equation}\label{initialdataassumption}
u_{0}\in J(\mathbb{R}^3)\cap \dot{H}^{\alpha}(\mathbb{R}^3)\cap\dot{\mathbb{B}}^{-1+\frac{3}{p}}_{p,\infty}.
\end{equation}
Furthermore, let $T$ be such that the strong solution $u$ associated with $u_{0}$ is defined on $\mathbb{R}^3\times (0,T)$. Then for any positive $\eta$, a constant $C$ exists such that, for any weak Leray solution $v$ associated with $v_{0}$, we have that if $\|v_{0}-u_{0}\|_{L_{2}}$ is small enough that
\begin{equation}\label{L2stab}
\frac{1}{2}\|v(t)-u(t)\|_{L_{2}}^2+\int\limits_{0}^{t} \|\nabla(v-u)(t')\|_{L_{2}}^2 dt'\leq C\|v_{0}-u_{0}\|_{L^{2}}^{2-2\eta}.
\end{equation}

\end{theorem}
 The heuristic idea of Jean-Yves Chemin is to split the strong solution $u$ into a low frequency part\footnote{For $j\geq 0$, the Fourier transform of $S_{j}u$ is compactly supported in $B(0,\frac{4}{3}2^{j})$.} $S_{j}u$ and a high frequency part. Then  $w_{j}:= v-S_{j}u$ satisfies the equation
 \begin{equation}\label{perturbedwithReynoldschemin}
 \partial_{t} w_{j}-\Delta w_{j}+ S_{j}u\cdot\nabla w_{j}+w_{j}\cdot\nabla S_{j}u+w_{j}\cdot\nabla w_{j}=-\nabla p_{j}-\nabla\cdot F_{j}
 \end{equation}
 \begin{equation}\label{icchemin}
 \textrm{div}\,w_{j}=0,\,\,\,\,w_{j}(\cdot,0)= v_{0}-S_{j}u_{0}
 \end{equation}
 \begin{equation}\label{Reynoldsstresschemin}
 R_{j}:= \nabla\cdot F_{j},\,\,\, F_{j}= S_{j}u \otimes S_{j}u- S_{j}(u\otimes u).
 \end{equation}
 A major part of Chemin's work involves paraproduct type analysis in frequency space to estimate the Reynolds stress $R_{j}$ of the strong solution $u$ with initial data $\dot{\mathbb{B}}^{-1+\frac{3}{p}}_{p,\infty}\cap L_{2}.$ He gets
 \begin{equation}\label{Reynoldsestchemin}
 \|F_{j}\|_{L_{2}(0,T; L_{2}(\mathbb{R}^3))}\leq C_{u,T} 2^{-\frac{j}{p-2}}.
 \end{equation}
 
 When one applies Gronwall's lemma (which can be done since $S_{j}u$ belongs to subcritical spaces) one gets
 \begin{equation}\label{Cheminenergyinequality}
 \|w_{j}(t)\|_{L_{2}}^2+\int\limits_{0}^{t}\|\nabla w_{j}(t')\|_{L_{2}}^2 dt'\leq C\|v_{0}\|_{L_{2}}^2+C\|u_{0}-S_{j}(u_{0})\|_{L_{2}}^2+ \int\limits_{0}^{t}\int\limits_{\mathbb{R}^3} S_{j}(u)\otimes w_{j}: \nabla w_{j} dxdt'.
 \end{equation}
 He then uses that $u$ 'just misses' by a logarithm  belonging to a `good' critical space (that allows Gronwall to be performed). In particular, $S_{j}(u)$ belongs to such `good' spaces but has a corresponding norm which grows like $\varepsilon j$ (for arbitrary $\varepsilon>0$) as the frequency parameter $j$ grows.
 Once the Gronwall argument is performed this produces 
 $$ \|w_{j}(t)\|_{L_{2}}^2\leq C(\|v_{0}\|_{L_{2}}^{2}+\|u_{0}-S_{j}(u_{0})\|_{L_{2}}^2+\|F_{j}\|_{L_{2}(0,T; L_{2})}^2)\exp(\varepsilon j).$$
The fact that $u_{0}\in \dot{H}^{\alpha}$  gives an exponential decay as $j$ grows of $\|u_{0}-S_{j}(u_{0})\|_{L_{2}}^2$. This, in conjunction with the exponential decay for the Reynold's stress \eqref{Reynoldsestchemin}, crucially offsets the small exponential growth coming from Gronwall's lemma. With a bit more work, this allows Chemin to conclude the proof of Theorem \ref{Chemin}.

 For the case $u_{0}=v_{0}$ the author has shown weak-strong uniqueness for wider yet classes of initial data than those considered by Chemin in \cite{chemin}. Specifically,  $u_{0}\in J(\mathbb{R}^3)\cap VMO^{-1}\cap \dot{B}^{s}_{q,q}$ with $s\in (-1+\frac{2}{q},0)$ (see also Lemarie-Rieusset \cite{LR19} for recent extensions for the class of `local Leray solutions'). For such classes of initial data, the author showed that
\begin{equation}\label{barkerweakstrong}
\|w(\cdot,t)\|_{L_{2}}^2\leq Ct^{\beta}.
\end{equation}
Such a decay depletes the singularity near the initial time due to $u$ having rough initial data. This allows us to infer that
$$\|w(\cdot,t)\|_{L^{2}_{x}}^2\leq C\int\limits_{0}^{t} \frac{(\sup_{0<s<T}s^{\frac{1}{2}}\|u(\cdot,s)\|_{L^{\infty}_{x}})^2}{s}\|w(\cdot,s)\|_{L^{2}_{x}}^2 ds .$$
Then the conclusion of weak-strong uniqueness in reached in \cite{barker2018} by a comparison of the quantity\footnote{ Comparison of this quantity was previously exploited by Dong and Zhang to prove weak-strong uniqueness results in \cite{dongzhang}.}
\begin{equation}\label{quantityweakstrong}
\frac{\|w(\cdot,t)\|_{L_{2}}^2}{t^{\beta}}.
\end{equation}
Unfortunately, such a strategy cannot prove Theorem \ref{maintheo}, since $w$ in that case isn't zero as $t\downarrow 0$ and hence has no decay to deplete the singularity in time of $u$. Hence despite weak-strong uniqueness being known for such initial data, the stronger result of $L_{2}$ stability remained open. In this paper we settle this case by means of Theorem \ref{maintheo}.
  
 \subsection{Novelty of our results}
The proof of Theorem \ref{maintheo} requires involves two observations that differ from from Chemin's proof of Theorem \ref{Chemin}

  The first observation is somewhat similar to the author's work on weak-strong uniqueness \cite{barker2018}, which was in turn inspired by the work of Calder\'{o}n \cite{Calderon90}. The difference compared to the authors work on weak-strong uniqueness  is that the space $\dot{H}^{\hat{\alpha}}$ was not used for the splittings, since $L_{2}$ instead played a prominent role. However, here this extra information must be kept to get the $L_{2}$ stability. Specifically if $u_{0}\in J(\mathbb{R}^3)\cap VMO^{-1}(\mathbb{R}^3)\cap  \dot{B}^{s}_{q,q}$ with $s\in (-1+\frac{2}{q},0)$, then we show that
$u_{0}=u_{0}^{1}+u_{0}^2$  with
$$u_{0}^{2} \in \dot{{B}}^{-1+\frac{3}{p}+\delta}_{p,p}(\mathbb{R}^3)\cap J(\mathbb{R}^3)\,\,\,\,\,\textrm{and}\,\,\,\,\,\,u_{0}^{1}\in\dot{H}^{\hat{\alpha}}\cap VMO^{-1}(\mathbb{R}^3)\cap J(\mathbb{R}^3).$$
We then reduce to considering $L_{2}$ stability of energy solutions  of the \textit{perturbed} Navier-Stokes equations.
\begin{equation}\label{perturbedNSE}
\partial_{t}U-\Delta U+U\cdot\nabla  U+ e^{t\Delta} u_{0}^{2}\cdot\nabla U+ U\cdot\nabla e^{t\Delta} u_{0}^{2}+\nabla P= -e^{t\Delta} u_{0}^{2}\cdot\nabla e^{t\Delta} u_{0}^{2}
\end{equation}
\begin{equation}\label{perturbedid}
\textrm{div}\,U=0,\,\,\,\,\,\,U(\cdot,0)=u_{0}^{1}
\end{equation}
In particular, the main goal reduces to showing an analogy of Theorem \ref{Chemin} but for this perturbed Navier-Stokes system and with initial data $u_{0}^{1}$.

 Recall that in Chemin's proof the exponential decay of the Reynold's stress \eqref{Reynoldsstresschemin} is crucial to offset the exponential growth in frequency parameter coming from estimates of the low frequency part $S_{j}u$ used for the application of Gronwall's lemma. However, notice that the estimate of the Reynold's stress \eqref{Reynoldsestchemin} does not possess any decay in $j$ as $p$ tends to infinity. Consequently, the main difficulty in proving Theorem \ref{maintheo} is that $u_{0}^{1}$ belongs to an $L_{\infty}$ based critical space $VMO^{-1}$. In particular, the arguments in \cite{chemin} seem to not give the required exponential decay for the Reynold's stress \eqref{Reynoldsstresschemin} for the spaces that $u_{0}^{1}$ belongs to.

 The second observation and main new idea of this paper is to overcome this difficulty by using additional information about the strong solution $U(\cdot, u^{1}_{0})$ which was not exploited in Chemin's paper. In particular, we use $u_{0}^{1}\in \dot{H}^{\hat{\alpha}}\cap J(\mathbb{R}^3)\cap VMO^{-1}(\mathbb{R}^3)$ to show\footnote{We mention that persistency arguments proven in \cite{PGLRpersistency} applied to strong solutions to the perturbed Navier-Stokes equations  would also suffice to show \eqref{strongsolutionpersistintro} with $\alpha=\hat{\alpha}$.  By comparison the Proposition we show does not require initial data that generates the existence of a local-in-time strong solution (such as $u_{0}^{1}\in VMO^{-1}$). Furthermore, the Proposition we give does not require that the solution $U$ is in critical spaces. } (see Proposition \ref{persistregularity}) that there exists $\alpha(\hat{\alpha},\delta)\in (0, \hat{\alpha}]$ such that

 \begin{equation}\label{strongsolutionpersistintro}
 U\in L_{\infty}(0,T; {H}^{\alpha}).
 \end{equation}
 This gives a decay of the $L^{\infty}_{t}L^{2}_{x}$ space-time norm involving the high frequencies of $U$, which gives that the associated Reynold's stress for the perturbed Navier-Stokes equations \eqref{perturbedNSE} has an exponential decay in $j$ depending on $\alpha$.

 We finally mention that it is crucial that  $e^{t\Delta}u_{0}^{2}$ belongs to subcritical spaces. In particular this means the extra terms in the perturbed Navier-Stokes equations do not destroy the arguments involving Gronwall's lemma. 
\subsection{Points of Independent Interest and Further Remarks}
\subsubsection{Partial Propagation of Regularity for the Perturbed Navier-Stokes Equations}
In Section 3.2,  we prove the following propagation of regularity result.

\begin{pro}\label{persistregularity}
Let $T>0$ be finite. Suppose that $V$ is divergence-free and
\begin{equation}\label{Vassumptionpersist}
V\in L^{\infty}_{T}L^{2},\,\,\,\,\sup_{0<t<T} t^{\frac{1}{2}(1-\delta)}\|V(\cdot,t)\|_{L^{\infty}_{x}}<\infty.
\end{equation}
Furthermore, suppose that there exists $\hat{\alpha}\in (0,1)$ such that 
\begin{equation}\label{u0persist}
u_{0}^{1}\in J(\mathbb{R}^3)\cap \dot{H}^{\hat{\alpha}}(\mathbb{R}^3)\cap \dot{B}^{-1}_{\infty,\infty}(\mathbb{R}^3).
\end{equation}
Assume that $U\in C_{w}([0,T]; J(\mathbb{R}^3))\cap L^{2}_{T} \dot{H}^{1}$ is a weak solution to the equation
\begin{equation}\label{Uequationpersist}
\partial_{t} U-\Delta U+V\cdot\nabla U+U\cdot\nabla V+U\cdot\nabla U+V\cdot\nabla V+ \nabla \Pi=0\,\,\,\textrm{in}\,\,\,\mathbb{R}^3\times (0,T)
\end{equation}
\begin{equation}\label{Uidpersist}
\textrm{div}\,U=0,\,\,\,U(\cdot,0)=u_{0}^{1}.
\end{equation}
Furthermore, assume  that $U$ satisfies the energy inequality for $t\in [0,T]$: 
\begin{equation}\label{Uenergyinequalitypersist}
 \|U(\cdot,t)\|_{L^{2}(\mathbb{R}^3)}^2+2\int\limits_{0}^{t}\int\limits_{\mathbb{R}^3} |\nabla U(y,s)|^2 dyds\leq \|u^{1}_{0}\|_{L^{2}(\mathbb{R}^3)}^2+ 2\int\limits_{0}^{t}\int\limits_{\mathbb{R}^3} (V\otimes U+ V\otimes V): \nabla U dyds.
 \end{equation}  
In addition, assume that $U$ satisfies
\begin{equation}\label{Ucritical}
\sup_{0<t<T} t^{\frac{1}{2}}\|U(\cdot,t)\|_{L^{\infty}(\mathbb{R}^3)}<\infty.
\end{equation}
Then the above assumptions allow us to conclude that there exists $\alpha(\hat{\alpha}, \delta)\in (0,\hat{\alpha}]$ such that
\begin{equation}\label{Usobolevpersist}
U\in L^{\infty}(0,T; {H}^{\alpha}(\mathbb{R}^3)).
\end{equation} 
\end{pro}
In \cite{PGLRpersistency} it is shown that when the initial data is in $VMO^{-1}$, the strong solution constructed by an iteration scheme propagates any additional regularity of the initial data on the homogeneous Besov scale. Furthermore, in \cite{vicolsilvestre} it is shown that when $b$ is divergence-free ad belongs to certain critical spaces\footnote{We say $\mathcal{X}\subset \mathcal{S}^{'}(\mathbb{R}^3\times \mathbb{R})$ is a critical space for the Navier-Stokes equations if $\|u\|_{\mathcal{X}}=\|u_{\lambda}\|_{\mathcal{X}}$ for any $\lambda>0$. Here $u_{\lambda}(x,t):= \lambda u(\lambda x, \lambda^2 t)$.}  that one can propagate the H\"{o}lder continuity of the initial data for the drift-diffusion  equation with pressure
\begin{equation}\label{driftdiffusepres}
\partial_{t}u-\Delta u+b\cdot\nabla u+\nabla q=0,\,\,\,\textrm{div}\,u=0,\,\,\,u(\cdot,0)=u^{1}_{0}.
\end{equation}
Although the proof of Proposition \ref{persistregularity} is concise and elementary, perhaps at first sight the statement seems somewhat unexpected. Indeed, it is not known if strong solutions  can be constructed for $u^{1}_{0}$ satisfying \eqref{u0persist}. Furthermore, the result of  Proposition \ref{persistregularity} even holds true when the assumption \eqref{Ucritical} is replaced by certain supercritical assumptions (see Remark \ref{Usupercritical}). Such propagation results may be of independent interest and of use in other contexts.
\subsubsection{Conjectures and remarks}
 Arguments from \cite{kochtataru} and the subsequent paper \cite{LRprioux} show that when $u_{0}\in J(\mathbb{R}^3)\cap VMO^{-1}(\mathbb{R}^3)$ there exists a $T(u_{0})$ and a weak Leray-Hopf solution $u(\cdot,u_{0})$ that is infinitely smooth on $\mathbb{R}^3\times (0,T(u_{0}))$. However, the classes of initial data for which weak-strong uniqueness is proven in \cite{barker2018} (and for which local Hadamard well-posedness is proven by means of Theorem \ref{maintheo})  \textit{just miss} the case $u_{0}\in J(\mathbb{R}^3)\cap VMO^{-1}(\mathbb{R}^3)$. In particular, $$ u_{0}\in J(\mathbb{R}^3)\cap VMO^{-1}(\mathbb{R}^3)= J(\mathbb{R}^3)\cap VMO^{-1}(\mathbb{R}^3)\cap \dot{B}^{s}_{q,q}$$
 with $s=-1+\frac{2}{q}$ and $q\in (2,\infty)$, whereas Theorem 1.3 in \cite{barker2018} assumes $s\in (-1+\frac{2}{q},0)$.
 Despite this, the following conjecture was made in \cite{barker2018}
 \begin{itemize}
 \item[] \textbf{(C)} If $u_{0}\in J(\mathbb{R}^3)\cap VMO^{-1}(\mathbb{R}^3)$, the associated weak Leray-Hopf solutions coincide on some time interval.
 \end{itemize}
 Let us recap reasoning from \cite{LR19} as to why such a conjecture seems plausible. For $u_{0}\in J(\mathbb{R}^3)$, Leray proved existence of at least one global-in-time weak Leray-Hopf solution by first considering the mollified system
 \begin{equation}\label{mollifiedsystem}
 \partial_{t} u_{\epsilon}-\Delta u_{\epsilon}+(\varphi_{\epsilon}\star u_{\epsilon})\cdot\nabla u_{\epsilon}+\nabla p_{\epsilon}=0
 \end{equation}
 \begin{equation}\label{mollifiedsysteminitialdata}
 \textrm{div}\, u_{\epsilon}=0,\,\,\,\,u_{\epsilon}(\cdot,0)= u_{0}.
 \end{equation}
 Here, $\varphi\in C_{0}^{\infty}(\mathbb{R}^3)$, $\int\limits_{\mathbb{R}^3} \varphi(x) dx=1$ and $\varphi_{\epsilon}(x):=\frac{1}{\epsilon^3}\varphi(\frac{x}{\epsilon})$. Then Leray uses energy estimates and compactness arguments to obtain global-in-time weak Leray-Hopf solutions in the limit as $\varepsilon\downarrow 0$.
 In \cite{LR19}, solutions obtained in such a way are called `\textit{restricted Leray solutions}'. Using arguments from \cite{LRprioux} (see also \cite{barker2018} for an exposition of those arguments), one gets that for $u_{0}\in VMO^{-1}\cap J(\mathbb{R}^3)$ there exists $\hat{T}(u_{0})$ such that for all $T\in (0, \hat{T})$ the following holds true. Namely, if $u(\cdot,u_{0})$ is a restricted Leray solution then
 \begin{equation}\label{restrcitedlerayest}
 \|u\|_{\mathcal{E}_{T}}\leq 2\|e^{t\Delta} u_{0}\|_{\mathcal{E}_{T}}.
  \end{equation}
 Here, 
\begin{equation}\label{pathspacenormdefintro}
\|u\|_{\mathcal{E}_{T}}:= \sup_{0<t<T} \sqrt{t}\|u(\cdot,t)\|_{L_{\infty}(\mathbb{R}^3)}+$$$$+\sup_{(x,t)\in \mathbb{R}^3\times ]0,T[}\Big(\frac{1}{|B(0,\sqrt{t})|}\int\limits_0^t\int\limits_{|y-x|<\sqrt{t}} |u|^2 dyds\Big)^{\frac{1}{2}}
\end{equation}
and $e^{t\Delta} u_{0}$ represents the heat-flow acting on $u_{0}$. Then \eqref{restrcitedlerayest} implies that \textbf{(C)} holds true for restricted Leray-solutions\footnote{For analogous statements in bounded domains for initial data in Besov spaces, we refer to \cite{FarwiggigahsuI}-\cite{Farwiggiga}.} by means of the uniqueness of mild solutions  constructed in \cite{kochtataru}.

 In some sense the above reasoning justifies why conjecture \textbf{(C)} is plausible for weak Leray-Hopf solutions. However, it appears to be an open problem to even show local Hadamard well-posedness when $\mathcal{Z}= VMO^{-1}$, for restricted Leray solutions. In turn this makes corresponding conjectures surrounding \textbf{(Q)} when $\mathcal{Z}=VMO^{-1}$ seemingly more speculative than \textbf{(C)}.  
 \begin{section}{Preliminaries}
\subsection{General Notation}

Throughout this paper we adopt the Einstein summation convention. For arbitrary vectors $a=(a_{i}),\,b=(b_{i})$ in $\mathbb{R}^{n}$ and for arbitrary matrices $F=(F_{ij}),\,G=(G_{ij})$ in $\mathbb{M}^{n}$ we put
 $$a\cdot b=a_{i}b_{i},\,|a|=\sqrt{a\cdot a},$$
 $$a\otimes b=(a_{i}b_{j})\in \mathbb{M}^{n},$$
 $$FG=(F_{ik}G_{kj})\in \mathbb{M}^{n}\!,\,\,F^{T}=(F_{ji})\in \mathbb{M}^{n}\!,$$
 $$F:G=
 F_{ij}G_{ij}\,\,\,\textrm{and}
 \,\,\,|F|=\sqrt{F:F}.$$
 Let $e^{t\Delta}u_{0}$ denote the heat kernel convoluted with $u_{0}$. 

  For ${\lambda}\in \mathbb{R}$, $\floor*{\lambda}$ denotes the greatest integer less than $\lambda$. Furthermore, $\ceil*{\lambda}$ denotes the smallest integer greater than $\lambda$. 
  
  If $X$ is a Banach space with norm $\|\cdot\|_{X}$, then $L_{s}(a,b;X)$, with $a<b$ and $s\in[1,\infty)$,  will denote the usual Banach space of strongly measurable $X$-valued functions $f(t)$ on $(a,b)$ such that
$$\|f\|_{L^{s}(a,b;X)}:=\left(\int\limits_{a}^{b}\|f(t)\|_{X}^{s}dt\right)^{\frac{1}{s}}<+\infty.$$ 
The usual modification is made if $s=\infty$.
Sometimes we will denote $L^{p}(0,T; L^{q})$ by $L^{p}_{T}L^{q}$ or $L^{p}(0,T; L^{p}_{x})$.
  
Let $C([a,b]; X)$ denote the space of continuous $X$ valued functions on $[a,b]$ with usual norm. In addition, let $C_{w}([a,b]; X)$ denote the space of $X$ valued functions, which are continuous from $[a,b]$ to the weak topology of $X$. 
\subsection{Function spaces} 
For a tempered distribution $f$, let  $$\mathcal{F}(f)(\xi):=\int\limits_{\mathbb{R}^3}\exp(-ix\cdot\xi) f(x) dx $$ denote its Fourier transform.
Let $d, m \in \N\setminus\{0\}$. We begin by recalling the definition of the \emph{homogeneous Besov spaces} $\dot B^s_{p,q}(\R^d;\R^m)$.  There exists a non-negative radial function $\varphi \in C^\infty(\R^d)$ supported on the annulus $\{ \xi \in \R^d : 3/4 \leq |\xi| \leq 8/3 \}$ and $\chi\in C^{\infty}_{0}(B(4/3))$ such that
\begin{equation}
	\chi(\xi)+\sum_{j\geq 0} \varphi(2^{-j} \xi) = 1, \quad \xi \in \R^3 ,
	\label{}
\end{equation}
\begin{equation}
	\sum_{j \in \Z} \varphi(2^{-j} \xi) = 1, \quad \xi \in \R^3 \setminus \{0\}.
	\label{}
\end{equation}
The  high frequency cut off $\dot{S}_{j}$ and the homogeneous Littlewood-Paley projectors $\dot \Delta_j$ are defined by
\begin{equation}
	\dot \Delta_j f = \varphi(2^{-j} D) f, \quad j \in \Z,
	\label{}
\end{equation}
\begin{equation}
	\dot S_j f = \chi(2^{-j} D) f, \quad j \in \Z,
	\label{}
\end{equation}
for all tempered distributions $f$ on $\R^d$ with values in $\R^m$. The notation $\varphi(2^{-j}D) f$ denotes convolution with the inverse Fourier transform of $\varphi(2^{-j}\cdot)$ with $f$. Notice that $\dot{S}_{j}= I-\sum_{k=j}^{\infty}\dot{\Delta}_{k}.$ Furthermore, for tempered distributions such that $\sum_{k\in\mathbb{Z}} \dot{\Delta}_{k}f$ converges to $f$ (in the sense of tempered distributions) we have that 
 $\dot{S}_{j}f= \sum_{k=-\infty}^{k=j-1} \dot{\Delta}_{k}f.$

 Let $p,q \in [1,\infty]$ and $s \in (-\infty,d/p)$.\footnote{The choice $s=d/p$, $q=1$ is also valid.} The homogeneous Besov space $\dot B^s_{p,q}(\R^d;\R^m)$ consists of all tempered distributions $f$ on $\R^d$ with values in $\R^m$ satisfying
	\begin{equation}
		\|{f}\|_{\dot B^s_{p,q}(\R^d;\R^m)} :=\Big(\sum_{j\in\mathbb{Z}} \big( 2^{js}\|\dot{\Delta}_{j} f\|_{L^{p}}\big)^{q}\Big)^{\frac{1}{q}}.
		\label{}
	\end{equation}
	and such that $\sum_{j \in \Z} \dot \Delta_j f$ converges to $f$ in the sense of tempered distributions on $\R^d$ with values in $\R^m$. In this range of indices, $\dot B^{s}_{p,q}(\R^d;\R^m)$ is a Banach space. When $s \geq 3/p$ and $q > 1$, the spaces must be considered \emph{modulo polynomials}. Note that other reasonable choices of the function $\varphi$ defining $\dot \Delta_j$ lead to equivalent norms. 

We now recall a particularly useful property of Besov spaces, i.e., their characterization in terms of the heat kernel. For all $s \in (-\infty,0)$, there exists a constant $c := c(s) > 0$ such that for all tempered distributions $f$ on $\R^3$,
		\begin{equation}
			c^{-1} \sup_{t > 0} t^{-\frac{s}{2}} \|e^{t\Delta} f\|_{L^p(\R^3)} \leq \|{f}\|_{\dot B^s_{p,\infty}(\R^3)} \leq c \sup_{t > 0} t^{-\frac{s}{2}}\|{e^{t\Delta} f}\|_{L^p(\R^3)}.
			\label{besovequivalentnorm}
		\end{equation}
		We will need the following Proposition, whose statement and proof can be found in the book \cite{bahourichemindanchin} (Proposition 2.22 there). In the Proposition below we use the notation
\begin{equation}\label{Sh}
\mathcal{S}_{h}^{'}:=\{ \textrm{ tempered\,\,distributions}\,\, u\textrm{\,\,\,such\,\,that\,\,} \lim_{j\rightarrow -\infty}\|S_{j}u\|_{L_{\infty}(\mathbb{R}^3)}=0\}.
\end{equation}
\begin{pro}\label{interpolativeinequalitybahourichemindanchin}
A constant $C$ exists with the following properties. If $s_{1}$ and $s_{2}$ are real numbers such that $s_{1}<s_{2}$ and $\theta\in ]0,1[$, then we have, for any $p\in [1,\infty]$ and any $u\in \mathcal{S}_{h}^{'}$,
\begin{equation}\label{interpolationactual}
\|u\|_{\dot{B}_{p,1}^{\theta s_{1}+(1-\theta)s_{2}}(\mathbb{R}^3)}\leq \frac{C}{s_2-s_1}\Big(\frac{1}{\theta}+\frac{1}{1-\theta}\Big)\|u\|_{\dot{B}_{p,\infty}^{s_1}(\mathbb{R}^3)}^{\theta}\|u\|_{\dot{B}_{p,\infty}^{s_2}(\mathbb{R}^3)}^{1-\theta}.
\end{equation}
\end{pro}
		Furthermore, we define the Chemin-Lerner norm\footnote{This was introduced in \cite{cheminlerner} for the special case $s=\frac{d}{2}+1$, $r=1$, $p=2$ and $q=2$.}
		\begin{equation}\label{cheminlerner}
		\|u\|_{\tilde{L^{r}_{T}}(\dot{B}^{s}_{p,q})}:=\|(2^{js}\|\dot{\Delta}_{j} u\|_{L^{r}(0,T; L^{p})})\|_{l^{q}(\mathbb{Z})}.
		\end{equation}
		The following useful Lemma was proven in \cite{bahourichemindanchin} (Lemma 2.4 there). We state in below.
		\begin{lemma}\label{localisationfourierest}
		Let $\mathcal{C}$ be an annulus. Positive constants $c$ and $C$ exist such that for all $p\in [1,\infty]$ and any couple $(t,\lambda)$ of positive real numbers, we have
		$$\textrm{supp}\, \hat{u}\subset \lambda \mathcal{C}\Rightarrow \|e^{t\Delta}u\|_{L_{p}}\leq Ce^{-c\lambda^2 t}\|u\|_{L_{p}}. $$
		\end{lemma}
		Lemma \ref{localisationfourierest} yields two useful estimates immediately. Namely, for $s\in\mathbb{R}$ and $p,q\in [1,\infty]^2$ we have
		\begin{equation}\label{cheminlernerheat}
		\|e^{t\Delta} u\|_{\tilde{L}^{r}_{T}(\dot{B}^{s+\frac{2}{r}}_{p,q})}\leq C(r,s,p,q)\|u\|_{\dot{B}^{s}_{p,q}}.
		\end{equation}
		Second,  by interpolation for homogeneous Sobolev spaces we have that for any $\alpha>0$
		\begin{equation}\label{sobolevsemigroup}
		\|e^{t\Delta} u\|_{\dot{H}^{\alpha}}\leq \frac{C'(\alpha)\|u\|_{L_{2}}}{t^{\frac{\alpha}{2}}}.
		\end{equation}
	 We will also make use of the following Lemma contained in the book \cite{bahourichemindanchin} (Corollary 2.54 there).
	 \begin{lemma}\label{besovspacesthatform analgebra}
	 Suppose that $(s,p,r)\in (0,\infty)\times [1,\infty]^2$ with $s<\frac{3}{p}$. Then there exists a  constant $C(s)$ such that
	 \begin{equation}\label{besovalgebraeqn}
	 \|uv\|_{\dot{B}^{s}_{p,q}(\mathbb{R}^3)}\leq \frac{C^{s+1}}{s}( \|v\|_{L_{\infty}(\mathbb{R}^3)}\|u\|_{\dot{B}^{s}_{p,q}(\mathbb{R}^3)}+\|u\|_{L_{\infty}(\mathbb{R}^3)}\|v\|_{\dot{B}^{s}_{p,q}(\mathbb{R}^3)})
	 \end{equation}
	 \end{lemma}
	 Finally, $BMO^{-1}(\mathbb{R}^3)$ is the space of all tempered distributions such that the following norm is finite:
\begin{equation}\label{bmo-1norm}
\|u\|_{BMO^{-1}(\mathbb{R}^3)}:=\sup_{x\in\mathbb{R}^3,R>0}\frac{1}{|B(0,R)|}\int\limits_0^{R^2}\int\limits_{B(x,R)} |e^{t\Delta}u|^2 dydt.
\end{equation}
Note that $VMO^{-1}(\mathbb{R}^3)$ is the subspace that coincides with the closure of test functions $C_{0}^{\infty}(\mathbb{R}^3)$, with respect to the norm (\ref{bmo-1norm}).

	 \begin{subsection}{Decompostions of Besov spaces}
	 Now, we can state  a Lemma regarding decomposition of homogeneous Besov spaces taken from the authors paper \cite{barker2018} (Proposition 2.8 there). For more general decomposition results, we refer to \cite{albrittonbarker}.
\begin{pro}\label{Decompgeneralbesov}
For $i=1,2,3$ let $p_{i}\in (1,\infty)$, $s_i\in \mathbb{R}$ and $\theta\in (0,1)$ be such that $s_1<s_0<s_2$ and $p_2<p_0<p_1$. In addition, assume the following relations hold:
\begin{equation}\label{sinterpolationrelationbesov}
s_1(1-\theta)+\theta s_2=s_0,
\end{equation}
\begin{equation}\label{pinterpolationrelationbesov}
\frac{1-\theta}{p_1}+\frac{\theta}{p_2}=\frac{1}{p_0}
\end{equation}
and
\begin{equation}\label{besovbanachconditionbesov}
{s_i}<\frac{3}{p_i}.
\end{equation}
Suppose that $u_0\in \dot{B}^{{s_{0}}}_{p_0,p_0}(\mathbb{R}^3).$
Then for all $\epsilon>0$, there exists $u^{1,\epsilon}\in \dot{B}^{s_{1}}_{p_1,p_1}(\mathbb{R}^3)$, $u^{2,\epsilon}\in \dot{B}^{s_{2}}_{p_2,p_2}(\mathbb{R}^3)$ such that 
\begin{equation}\label{udecompgeneralbesov}
u= u^{1,\epsilon}+u^{2,\epsilon},
\end{equation}
\begin{equation}\label{u_1estgeneralbesov}
\|u^{1,\epsilon}\|_{\dot{B}^{s_{1}}_{p_1,p_1}}^{p_1}\leq \epsilon^{p_1-p_0} \|u_0\|_{\dot{B}^{s_{0}}_{p_0,p_0}}^{p_0}
\end{equation}
and
\begin{equation}\label{u_2estgeneralbesov}
\|u^{2,\epsilon}\|_{\dot{B}^{s_{2}}_{p_2,p_2}}^{p_2}\leq C(s_1,s_2,p_0,p_1,p_2, \|\mathcal{F}^{-1}\varphi\|_{L_1})\epsilon^{p_2-p_0} \|u_0\|_{\dot{B}^{s_{0}}_{p_0,p_0}}^{p_0}
.\end{equation}
\end{pro}
The following corollary is essentially contained in the author's paper \cite{barker2018}. However, in that case we did not show that one piece of the initial data had better regularity than $L_{2}$. For completeness, we therefore provide further details.
\begin{cor}\label{Decomp}

 Suppose that $q>3$,
\begin{equation}\label{initialdataassumptiondecomp}
u_0\in \dot{B}^{s}_{q,q}(\mathbb{R}^3)\cap L_{2}(\mathbb{R}^3)\,\,\,\textrm{with}\,\,\,\,s\in (-1+\frac{2}{q},0)
\end{equation}
and
 $\textrm{div}\,\,u_0=0$ in the sense of distributions. \\
  Then the above assumption imply that there exists $\max{(q,4)}<p<\infty$, $\delta\in (0,1-\frac{3}{p})$ and $\hat{\alpha}\in (0,\frac{3}{2})$ such that  for any $\epsilon>0$ there exists weakly divergence-free functions 
$\bar{u}^{1,\epsilon}\in \dot{B}^{_1+\frac{3}{p}+\delta}_{p,p}(\mathbb{R}^3)\cap L_{2}(\mathbb{R}^3)$ and $\bar{u}^{2,\epsilon}\in \dot{H}^{\hat{\alpha}}(\mathbb{R}^2)\cap L_2(\mathbb{R}^3)$ such that 
 
\begin{equation}\label{udecomp1}
u_0= \bar{u}^{1,\epsilon}+\bar{u}^{2,\epsilon},
\end{equation}
\begin{equation}\label{baru_1est}
\|\bar{u}^{1,\epsilon}\|_{\dot{B}^{-1+\frac{3}{p}+\delta}_{p,p}}^{p}\leq \epsilon^{p-q} \|u_0\|_{\dot{B}^{s}_{q,q}}^{q},
\end{equation}
\begin{equation}\label{baru_2est}
\|\bar{u}^{2,\epsilon}\|_{\dot{H}^{\hat{\alpha}}(\mathbb{R}^3)}^2\leq C(s,\alpha,p,q,\|\mathcal{F}^{-1}\varphi\|_{L_1}) \epsilon^{2-q}\|u_0\|_{\dot{B}^{s}_{q,q}}^q
\end{equation}
and  
\begin{equation}\label{baru_1est.1}
\|\bar{u}^{2,\epsilon}\|_{L_2}, \|\bar{u}^{1,\epsilon}\|_{L_2}\leq  C(\|\mathcal{F}^{-1}\varphi\|_{L_1})\|u_0\|_{L_{2}}.
\end{equation}

\end{cor}
\begin{proof}
Take $p>\max(q,4)$. The assumption that $s\in (-1+\frac{2}{q},0)$ implies that there exists $\tilde{\varepsilon}\in (0, 1-\frac{2}{q})$ such that
\begin{equation}\label{sexpression}
s:=-1+\frac{2}{q}+\tilde{\varepsilon}.
\end{equation} 
Then, 
$$\frac{\frac{2}{p}-1}{\frac{2}{q}-1}s-(-1+\frac{3}{p})=-\frac{1}{p}+\frac{-\frac{2}{p}+1}{-\frac{2}{q}+1}\tilde{\varepsilon}.$$
So
$$lim_{p\uparrow\infty}\Big(\frac{\frac{2}{p}-1}{\frac{2}{q}-1}s-(-1+\frac{3}{p})\Big)=\frac{1}{-\frac{2}{q}+1}\tilde{\varepsilon}>0. $$
Thus, there exists a $p$ sufficiently large and a $\hat{\delta}>0$ such that
\begin{equation}\label{indicerelationsubcritical1}
\frac{1-\frac{2}{p}}{1-\frac{2}{q}}s=-1+\frac{3}{p}+\hat{\delta}.
\end{equation}
Referring to the previous proposition, let $p_{0}=q$, $p_{1}=p$ and $p_{2}=2$ and let $\theta$ be such that 
$$\frac{1-\theta}{p}+\frac{\theta}{2}=\frac{1}{q}.$$
Thus $$1-\theta= \frac{1-\frac{2}{q}}{1-\frac{2}{p}}. $$
Thus, \eqref{indicerelationsubcritical1} implies
\begin{equation}\label{indicerelationsubcriticalmain}
s=(1-\theta)\Big(-1+\frac{3}{p}+\hat{\delta}\Big)=(1-\theta)\Big(-1+\frac{3}{p}+{\hat{\delta}}-\frac{\theta\hat{\alpha}}{1-\theta}\Big)+\theta\hat{\alpha}.
\end{equation}
Choose
\begin{equation}\label{hatalphachoice}
\hat{\alpha}\in \Big(0,\min\Big(\frac{3}{2}, \frac{\hat{\delta}(1-\theta)}{\theta}\Big)\Big)
\end{equation}
Now, define $\delta:= \hat{\delta}-\frac{\theta\hat{\alpha}}{1-\theta}>0$, $s_{1}:= -1+\frac{3}{p}+\delta$, $s_{0}=s$ and $s_{2}:=\hat{\alpha}.$
The above relations allow us  to apply Proposition \ref{Decompgeneralbesov} to obtain the following decomposition:
(we note that for $\hat{\alpha}<\frac{3}{2}$, $\dot{B}^{\hat{\alpha}}_{2,2}(\mathbb{R}^3)$ coincides with $\dot{H}^{\hat{\alpha}}(\mathbb{R}^3)$ with equivalent norms)
\begin{equation}\label{udecompchapter2}
u_0= {u}^{1,\epsilon}+{u}^{2,\epsilon},
\end{equation}
\begin{equation}\label{u_1estchapter2}
\|{u}^{1,\epsilon}\|_{\dot{B}^{-1+\frac{3}{p}+\delta}_{p,p}}^{p}\leq \epsilon^{p-q} \|u_0\|_{\dot{B}^{s}_{q,q}}^q
\end{equation}
and
\begin{equation}\label{u_2estchapter2}
\|{u}^{2,\epsilon}\|_{\dot{H}^{\hat{\alpha}}}^2\leq C(\alpha,p,q,\|\mathcal{F}^{-1}\varphi\|_{L_1})\epsilon^{2-q} \|u_0\|_{\dot{B}^{s}_{q,q}}^q
.
\end{equation} 

For $j\in\mathbb{Z}$ and $m\in\mathbb{Z}$, it can be seen that  \begin{equation}\label{besovpersistency}
\|\dot{\Delta}_{m}\left( (\dot{\Delta}_{j}u_{0})\chi_{|\dot{\Delta}_{j}u_{0}|\geq N(j,\epsilon)}\right)\|_{L_{2}},\,\|\dot{\Delta}_{m}\left( (\dot{\Delta}_{j}u_{0})\chi_{|\dot{\Delta}_{j}u_{0}|\leq N(j,\epsilon)}\right)\|_{L_{2}}\leq C(\|\mathcal{F}^{-1} \varphi\|_{L_1})) \|\dot{\Delta}_{j} u_0\|_{L_2}.
\end{equation}
It is known that  $u_0\in L_{2}$ implies
$$ \|u_0\|_{L_2}^2=\sum_{j\in\mathbb{Z}}\|\dot{\Delta}_{j} u_0\|_{L_2}^2.$$
Using this, (\ref{besovpersistency}) and the expression of $u^{1,\epsilon}$ given by Proposition 2.8 of the author's paper \cite{barker2018} , we can infer that
$$\|{u}^{2,\epsilon}\|_{L_2},\,\|{u}^{1,\epsilon}\|_{L_2}\leq   C(\|\mathcal{F}^{-1} \varphi\|_{L_1})\|u_0\|_{L_{2}}.$$

The Leray projector $\mathbb{P}$, which projects onto divergence free vector fields, is defined as
$$\mathbb{P}f:= f+\nabla(-\Delta)^{-1}(\textrm{div}\,f).$$ To establish the decomposition of the Corollary, we
 apply the Leray projector to each of $u^{1,\epsilon}$ and $u^{2,\epsilon}$, which is a continuous linear operator on the homogeneous  Besov spaces under consideration.

\end{proof}
	 \end{subsection}
	
\end{section}
\begin{section}{Properties of strong solutions of the Navier-Stokes equations}
In this section, we first describe certain properties of the strong solution $u(\cdot,u_{0})$ with initial data $u_{0}\in VMO^{-1}(\mathbb{R}^3)\cap J(\mathbb{R}^3)\ \cap \dot{B}^{s}_{q,q}(\mathbb{R}^3)$ ($s\in (-1+\frac{2}{q},0)$) that will be needed to prove Theorem \ref{maintheo}. We must mention that the first subsection is mostly a collection of results already contained in the literature, gathered for the reader's convenience.
Where the context is slightly different to the previous literature, or when fixes are needed, we provide the reader with details.

In the second subsection we prove Proposition \ref{persistregularity}. This will imply that some of the Sobolev regularity of $u_{0}^{1}$ persists for $U(x,t):= u- e^{t\Delta}u_{0}^{2}$. This will be a crucial ingredient in proving Theorem \ref{maintheo}.

\begin{subsection}{Regularity properties of strong solutions}
First, we discuss the construction and regularity of the strong solution $u$ described in Theorem \ref{maintheo}. For initial data in $VMO^{-1}(\mathbb{R}^3)$, local-in-time strong solutions to the Navier-Stokes equations were shown to exist by Koch and Tataru in \cite{kochtataru}. Such solutions belong to the pathspace $\mathcal{P}_{T}$ where
\begin{equation}\label{pathspaceBMO-1}
\mathcal{P}_{T}:=\{ u\in \mathcal{S}^{'}(\mathbb{R}^3\times \mathbb{R}_{+}): \|u\|_{\mathcal{E}_{T}}<\infty\}.
\end{equation}
Here, 
\begin{equation}\label{pathspacenormdef}
\|u\|_{\mathcal{E}_{T}}:= \sup_{0<t<T} \sqrt{t}\|u(\cdot,t)\|_{L_{\infty}(\mathbb{R}^3)}+$$$$+\sup_{(x,t)\in \mathbb{R}^3\times ]0,T[}\Big(\frac{1}{|B(0,\sqrt{t})|}\int\limits_0^t\int\limits_{|y-x|<\sqrt{t}} |u|^2 dyds\Big)^{\frac{1}{2}}.
\end{equation}
Furthermore, the Koch-Tataru solutions satisfy the integral formulation of the Navier-Stokes equations
\begin{equation}\label{mildformulationofNSE}
u(x,t):= e^{t\Delta}u_{0}+\int\limits_{0}^{t} e^{(t-s)\Delta}\mathbb{P}\nabla\cdot(u\otimes u) ds.
\end{equation}
Here, $e^{t\Delta}$ denotes the heat semigroup in $\mathbb{R}^3$ and $\mathbb{P}$ denotes the projection of vector fields onto divergence-free vector fields. Throughout this paper, we denote the bilinear term by
\begin{equation}\label{bilinearterm}
B(f,g):=\int\limits_{0}^{t} e^{(t-s)\Delta}\mathbb{P}\nabla\cdot(f\otimes g) ds.
\end{equation}

From (\ref{bmo-1norm}), we see that for $0<T\leq\infty$
 \begin{equation}\label{BMO-1embeddingcritical}
u_0 \in BMO^{-1}(\mathbb{R}^3)\Rightarrow \|S(t)u_{0}\|_{\mathcal{E}_{T}}\leq C\|u_0\|_{BMO^{-1}}.
\end{equation}
Since $C_{0}^{\infty}(\mathbb{R}^3)$ is dense in $VMO^{-1}(\mathbb{R}^3)$, we can see from the above that for $u_0\in VMO^{-1}(\mathbb{R}^3)$
\begin{equation}\label{VMO-1shrinkingchapter2}
\lim_{T\rightarrow 0^{+}} \|S(t)u_{0}\|_{\mathcal{E}_{T}}=0.
\end{equation}
It was shown in \cite{kochtataru} that there exists a universal constant $C$ such that for all $f,\,g\in \mathcal{E}_{T}$ 
\begin{equation}\label{bilinbmo-1}
\|B(f,g)\|_{\mathcal{E}_{T}}\leq C\|f\|_{\mathcal{E}_{T}}\|g\|_{\mathcal{E}_{T}}.
\end{equation}
Here is the needed proposition related to the construction of the `strong solution'. The statement and references can also be found in \cite{barker2018}.
\begin{pro}\label{regularitycriticalbmo-1} 
Suppose that $u_0\in VMO^{-1}(\mathbb{R}^3)\cap J(\mathbb{R}^3).$ 
There exists a  universal constant $\varepsilon_0>0$ such that 
if 
\begin{equation}\label{smallnessasmpbmo-1}
\|S(t)u_{0}\|_{\mathcal{E}_{T}}<\varepsilon_0,
\end{equation}
then 
there exists a $u\in \mathcal{E}_{T}$, which solves the Navier-Stokes equations in the sense of distributions and satisfies the following properties.
The first property is that $u$ solves the following integral equation:
\begin{equation}\label{vintegeqnbmo-1}
u(x,t):= S(t)u_{0}+B(u,u)(x,t)
\end{equation}
in $\mathbb{R}^3\times (0,T)$, along with the estimate
\begin{equation}\label{vintegestbmo-1}
\|u\|_{\mathcal{E}_{T}}<2\|S(t)u_0\|_{\mathcal{E}(T)}.
\end{equation}
 The second property is that $u$ is a weak Leray-Hopf solution on $\mathbb{R}^3\times (0,T)$. 

 If $\pi_{u\otimes u}$ is the associated pressure we have (here, $\lambda\in (0,T)$ and $p\in (2,\infty)$):
\begin{equation}\label{presspacebmo-1}
\pi_{u\otimes u} \in L^{\frac{5}{3}}(\mathbb{R}^3\times (0,T))\cap L^{\infty}( \lambda, T; L^{\frac{p}{2}}(\mathbb{R}^3)).
\end{equation}
Furthermore for $\lambda\in(0,T)$ and $k=0,1\ldots$, $l=0,1\ldots$:
\begin{equation}\label{vpsmoothbmo-1}
\sup_{(x,t)\in \mathbb{R}^3\times (\lambda,T)}|\partial_{t}^{l}\nabla^{k} u|+|\partial_{t}^{l}\nabla^{k} \pi_{u\otimes u}|\leq c(p_0,\lambda,\|u_0\|_{BMO^{-1}},\|u_0\|_{L_{2}},k,l ).
\end{equation}
Finally, when $\varepsilon_{0}$ is sufficiently small there exists constants $C(k)$ such that for $k=0,1,\ldots$
\begin{equation}\label{CKNsmoothing}
\sup_{0<t<T} (\sqrt{t})^{k+1}\|\nabla^{k}u(\cdot,t)\|_{L_{\infty}(\mathbb{R}^3)}\leq C(k).
\end{equation}
\end{pro}
\begin{proof}
Conclusions \eqref{vintegeqnbmo-1} and \eqref{vintegestbmo-1} are due to \cite{kochtataru}.
The fact that $u$ is a weak-Leray Hopf solution follows from ideas in \cite{LRprioux} (see also the appendix of \cite{barker2018}). The proof of \eqref{presspacebmo-1}-\eqref{vpsmoothbmo-1} are also described in \cite{barker2018}.

Let us focus on proving \eqref{CKNsmoothing}, which will be required to prove Theorem \ref{maintheo}. 
First, we recall the known fact that $\pi_{u\otimes u}$ is a composition of Riesz transforms acting on $u\otimes u$. Thus, using \eqref{vintegestbmo-1} and the Caldero\'{o}n-Zygmund theory we get that for $r\in (0,\sqrt{t})$
\begin{equation}\label{pressurebmo}
\|\pi_{u\otimes u}\|_{L^{\infty}(t-r^2,t; BMO(\mathbb{R}^3)}\leq C\|u\|_{L^{\infty}(\mathbb{R}^3\times (t-r^2,t)}^2\leq \frac{C'\varepsilon_{0}^{2}}{t-r^2}.
\end{equation}
Using this, we obtain that for all $x\in\mathbb{R}^3$
\begin{equation}\label{CKNexpression}
\frac{1}{r^2}\int\limits^{t}_{t-r^2}\int\limits_{B(x,r)} |u|^3+|\pi_{u\otimes u}-(\pi_{u\otimes u})_{B(x,r)}|^{\frac{3}{2}} dxdt'\leq \frac{C''r^3\varepsilon_{0}^{3}}{(t-r^2)^{\frac{3}{2}}}.
\end{equation}
Here, $C''$ is a universal constant. Taking $r:=\frac{\sqrt{t}}{2}$ we get
\begin{equation}\label{CKNexpression1}
\frac{4}{t}\int\limits^{t}_{t-\frac{t}{4}}\int\limits_{B(x,\frac{\sqrt{t}}{2})} |u|^3+|\pi_{u\otimes u}-(\pi_{u\otimes u})_{B(x,r)}|^{\frac{3}{2}} dxdt'\leq {C'''\varepsilon_{0}^{3}}.
\end{equation}
If $C'''\varepsilon_{0}^{3}\leq\varepsilon_{CKN}$, we can apply the Caffarelli-Kohn-Nirenberg theory \cite{CKN} to immediately infer \eqref{CKNsmoothing}.
\end{proof}
Next will discuss some further regularity properties of Koch and Tataru's strong solution that will be needed to prove Theorem \ref{maintheo}. First we begin with a lemma taken from \cite{zhangzhang} (Lemma 2.1 there).
\begin{lemma}\label{zhangzhangfreqlocalisedkernels}
Let $q\in\mathbb{Z}$, $(x,t)\in \mathbb{R}^3\times (0,\infty)$, $i,j,n\in\{1,2,3\}$ and $\varphi$ be as in the definition of the Littlewood Paley projectors\footnote{See section 2.2 `Function Spaces'}. Define
\begin{equation}\label{zhangkernel1}
g_{q}^{i,j,n}(x,t):=\frac{1}{8\pi^3}\int\limits_{\mathbb{R}^3} e^{ix\cdot\xi}\varphi(2^{-q}\xi)e^{-t|\xi|^2}\Big(\delta_{ij}-\frac{\xi_{i}\xi_{j}}{|\xi|^2}\Big)\xi_{n} d\xi
\end{equation}
\begin{equation}\label{zhangkernel2}
g_{1,q}^{i,j}(x,t):=\frac{1}{8\pi^3}\int\limits_{\mathbb{R}^3} e^{ix\cdot\xi}\varphi(2^{-q}\xi)e^{-t|\xi|^2}\Big(\delta_{ij}-\frac{\xi_{i}\xi_{j}}{|\xi|^2}\Big) d\xi
\end{equation}
\begin{equation}\label{zhangkernel3}
g_{2,q}(x,t):=\frac{1}{8\pi^3}\int\limits_{\mathbb{R}^3} e^{ix\cdot\xi}\varphi(2^{-q}\xi)e^{-t|\xi|^2}d\xi.
\end{equation}
Then, we have the estimates
\begin{equation}\label{firstkernelestZhang}
|g^{i,j,n}_{q}(x,t)|\leq\frac{C_{univ}2^{4q}e^{-ct2^{2q}}}{1+(2^{q}|x|)^{6}}
\end{equation}
and
\begin{equation}\label{secondkernelestZhang}
|g^{i,j}_{1,q}(x,t)|+|g_{2,q}(x,t)|\leq\frac{C_{univ}2^{3q}e^{-ct2^{2q}}}{1+(2^{q}|x|)^{6}}
\end{equation}
\end{lemma}
Now, we state the main further regularity properties of the solutions in Proposition \ref{regularitycriticalbmo-1}.
\begin{pro}\label{zhangzhangregularity}
Let $u$ be the strong solution as in Proposition \ref{regularitycriticalbmo-1}, with initial data $u_{0}\in VMO^{-1}(\mathbb{R}^3)\cap J(\mathbb{R}^3)$. Then one also has that
\begin{equation}\label{zhangLinfinityendpoint}
\|u\|_{L^{\infty}(0,T; \dot{B}^{-1}_{\infty,\infty})}\leq C_{univ}(\|u_{0}\|_{BMO^{-1}}+\|u\|_{\mathcal{E}_{T}}^{2}),
\end{equation}
\begin{align}\label{zhangL1B1}
\begin{split}
\|u\|_{\tilde{L}^{1}(0,T; \dot{B}^{1}_{\infty,\infty})} &\leq C_{univ}(\|u\|_{\mathcal{E}_{T}}+\|e^{t\Delta}u_{0}\|_{\mathcal{E}_{T}}+\|u\|_{\mathcal{E}_{T}}\sup_{0<s<T} s\|\nabla u(\cdot,s)\|_{L_{\infty}})+\|e^{t\Delta}u_{0}\|_{\tilde{L}^{1}(0,T; \dot{B}^{1}_{\infty,\infty})}\\&
\leq C_{univ}'(\|e^{t\Delta}u_{0}\|_{\mathcal{E}_{T}}+\|e^{t\Delta}u_{0}\|_{\tilde{L}^{1}(0,T; \dot{B}^{1}_{\infty,\infty})}).
\end{split}
\end{align}
\end{pro}
\begin{proof}
The proof of \eqref{zhangLinfinityendpoint} is proven in \cite{zhangzhang} (Proposition 2.2 there) for the case when $T=\infty$. The adjustment for $T$ finite is not difficult and we omit it. In \cite{zhangzhang}, a proof of \eqref{zhangL1B1} is also presented for $T=\infty$ (Proposition 2.3 there) but there seems to be a minor error in the argument. We find it instructive to present their arguments here, but with the minor fix.

From the definition of the Chemin-Lerner spaces \eqref{cheminlerner}, we need to show
\begin{align}
\begin{split}
&\sup_{q\in\mathbb{Z}} 2^{q}\int\limits_{0}^{T}\|\dot{\Delta}_{q}u(\cdot,t)\|_{L_{\infty}} dt \leq C_{univ}(\|u\|_{\mathcal{E}_{T}}+\|e^{t\Delta}u_{0}\|_{\mathcal{E}_{T}}+\|u\|_{\mathcal{E}_{T}}\sup_{0<s<T} s\|\nabla u(\cdot,s)\|_{L_{\infty}})\\&+\|e^{t\Delta}u_{0}\|_{\tilde{L}^{1}(0,T; \dot{B}^{1}_{\infty,\infty})}.
\end{split}
\end{align}
For the case of $q$ such that  $2^{-2q}\geq T$, we have
\begin{equation}\label{zhangregeasy}
2^{q}\int\limits_{0}^{T}\|\dot{\Delta}_{q}u(\cdot,t)\|_{L_{\infty}} dt \leq 2^{q}\|u\|_{\mathcal{E}_{T}} \int\limits_{0}^{2^{-2q}} \frac{1}{t^{\frac{1}{2}}} dt\leq C\|u\|_{\mathcal{E}_{T}}.
\end{equation}
We are left to consider the case $T>2^{-2q}$ and we see that for this case, it suffices to show
\begin{equation}\label{zhangregharder}
2^{q}\int\limits_{2^{-2q}}^{T}\|\dot{\Delta}_{q}B(u,u)(\cdot,t)\|_{L_{\infty}} dt \leq C_{univ}(\|u\|_{\mathcal{E}_{T}}+\|e^{t\Delta}u_{0}\|_{\mathcal{E}_{T}}+\|u\|_{\mathcal{E}_{T}}\sup_{0<s<T} s\|\nabla u(\cdot,s)\|_{L_{\infty}}).
\end{equation}
Following the proof given in \cite{zhangzhang}, we have that
\begin{align}
\begin{split}\label{bilinearfreqlocalised}
&\dot{\Delta}_{q} B(u,u)(x,t)= \int\limits_{0}^{\frac{t}{2}} \dot{\Delta}_{q} e^{(t-s)\Delta}\mathbb{P}\nabla\cdot(u\otimes u)ds+\int\limits_{\frac{t}{2}}^{t}\int\limits_{\mathbb{R}^3}  g_{1,q}(y,t-s)\nabla\cdot(u(x-y,s)\otimes u(x-y,s))dsdy\\&=
F_{1,q}(x,t)+F_{2,q}(x,t).
\end{split}
\end{align}
We treat $F_{1,q}$ in the same way as in \cite{zhangzhang}. First observe that 
$$F_{1,q}(x,t)= \int\limits_{\mathbb{R}^3} g_{2,q}(y, \frac{t}{2})(u(x-y,\frac{t}{2})-e^{\frac{t}{2}\Delta}u_{0}(x-y)) dy. $$
Using Lemma \ref{zhangzhangfreqlocalisedkernels}, we get that
\begin{align}
\begin{split}\label{zhangF1qest}
&2^{q}\int\limits_{2^{-2q}}^{T}\|F_{1,q}(\cdot,t)\|_{L^{\infty}} dt\leq C_{univ} 2^{q}(\|u\|_{\mathcal{E}_{T}}+\|e^{t\Delta}u_{0}\|_{\mathcal{E}_{T}})\int\limits_{2^{-2q}}^{T}\int\limits_{\mathbb{R}^3}\frac{C_{univ}2^{3q}e^{-ct2^{2q}}}{t^{\frac{1}{2}}(1+(2^{q}|y|)^{6})} dydt\\&\leq C_{univ} (\|u\|_{\mathcal{E}_{T}}+\|e^{t\Delta}u_{0}\|_{\mathcal{E}_{T}}).
\end{split}
\end{align}
The  very minor fix required for the proof in \cite{zhangzhang} regards $F_{2,q}$.  In \cite{zhangzhang}, the integral is split into the regions $B(0, 2\sqrt{t})$ and $\mathbb{R}^3\setminus B(0, 2\sqrt{t})$. In the outer region the authors in \cite{zhangzhang} integrate by parts but seem to not account for the boundary traces on the sphere of radius $2\sqrt{t}$. The very minor fix we propose is to not integrate by parts and then to proceed with similar arguments to the estimates used in \cite{zhangzhang} for the integral over the ball $B(0, 2\sqrt{t})$.

Indeed, using Lemma \ref{zhangzhangfreqlocalisedkernels} we get that
\begin{align}
\begin{split}\label{zhangF1qestfinal}
&2^{q}\int\limits_{2^{-2q}}^{T}\|F_{2,q}(\cdot,t)\|_{L^{\infty}} dt\leq C_{univ} 2^{q}(\|u\|_{\mathcal{E}_{T}}\sup_{0<s<T} s\|\nabla u(\cdot,s)\|_{L_{\infty}})\int\limits_{2^{-2q}}^{T}\int\limits_{\frac{t}{2}}^{t}\int\limits_{\mathbb{R}^3}\frac{C_{univ}2^{3q}e^{-c(t-s)2^{2q}}}{s^{\frac{3}{2}}(1+(2^{q}|y|)^{6})} dydsdt\\&\leq C_{univ} \|u\|_{\mathcal{E}_{T}}\sup_{0<s<T} s\|\nabla u(\cdot,s)\|_{L_{\infty}}.
\end{split}
\end{align}
This completes the proof.
\end{proof}
\begin{remark}\label{VMO-1heat}
Using that $BMO^{-1}$ is continuously embedded into $\dot{B}^{-1}_{\infty,\infty}$, we get from \eqref{cheminlernerheat} that
\begin{equation}\label{BMO-1cheminlerner}
 \|e^{t\Delta} u_{0}\|_{\tilde{L}^{1}(\dot{B}^{1}_{\infty,\infty})},\,\|e^{t\Delta} u_{0}\|_{{L}^{\infty}(\dot{B}^{-1}_{\infty,\infty})}\leq \|u_{0}\|_{BMO^{-1}}.
 \end{equation}
Furthermore, if $u_{0}$ is smooth and compactly supported we have $$  \|e^{t\Delta} u_{0}\|_{\tilde{L}^{1}_{T}(\dot{B}^{1}_{\infty,\infty})}= \sup_{q\in\mathbb{Z}}\Big(2^{q} \int\limits_{0}^{T} \|\dot{\Delta}_{q} e^{t\Delta}u_{0}\|_{L^{\infty}(\mathbb{R}^3)}dt\Big)\leq cT\| u_{0}\|_{\dot{B}^{1}_{\infty,\infty}}.$$
Using this, \eqref{BMO-1cheminlerner} and a density argument allows us to infer that
\begin{equation}\label{BMO-1cheminlernershrinking}
 u_{0}\in VMO^{-1}\Rightarrow \lim_{T\downarrow 0}\|e^{t\Delta} u_{0}\|_{\tilde{L}^{1}_{T}(\dot{B}^{1}_{\infty,\infty})}=0.
 \end{equation}
\end{remark}
\end{subsection}
\begin{subsection}{Propagation of Sobolev Regularity  for the Navier-Stokes Equations}
\begin{lemma}\label{lemmalinearsobolevest}
Define 
\begin{equation}\label{Lfdef}
L(f)(\cdot,t):= \int\limits_{0}^{t} e^{(t-s)\Delta} \mathbb{P}\nabla\cdot f(\cdot,s) ds.
\end{equation}
Assume that for some finite $T>0$ and $p\in(2,\infty)$ 
\begin{equation}\label{fassumption}
f\in L^{p}_{T}L^{2}_{x}.
\end{equation}
Then for every $\beta\in (0, 1-\frac{2}{p})$ we have
\begin{equation}\label{Lfest}
\|L(f)\|_{L^{\infty}_{T} H^{\beta}}\leq C(\beta,T)\max\Big(\|f\|_{L^{p}_{T}L^{2}_{x}},\|f\|_{L^{2}_{T}L^{2}_{x}} \Big)
\end{equation}
\end{lemma}
\begin{proof}
It is classical that
\begin{equation}\label{LfestL2}
\|L(f)\|_{L^{\infty}_{T} L^{2}_{x}}\leq 2\|f\|_{L^{2}_{T}L^{2}_{x}}.
\end{equation}
Applying \eqref{sobolevsemigroup} and using the continuity of the Leray Projector $\mathbb{P}$ on homogeneous Sobolev spaces, we get that
$$\|e^{(t-s)\Delta}\mathbb{P}\nabla\cdot f(\cdot,s)\|_{\dot{H}^{\beta}}\leq C\|e^{(t-s)\Delta} f(\cdot,s)\|_{\dot{H}^{\beta+1}}\leq \frac{C'\|f(\cdot,s)\|_{L^{2}_{x}}}{(t-s)^{\frac{1+\beta}{2}}}. $$
 We then apply H\"{o}lder's inequality in time to get 
$$\|L(f)(\cdot,t)\|_{\dot{H}^{\beta}}\leq \|f\|_{L^{p}_{T}L^{2}_{x}}\Big(\int_{0}^{t} \frac{ds}{(t-s)^{\frac{(1+\beta)p}{2(p-1)}}}\Big)^{1-\frac{1}{p}}. $$
This converges for  $p\in (2,\infty)$ and $\beta\in (0,1-\frac{2}{p})$.
\end{proof}
\begin{lemma}\label{estimatenearinitialtime}
 Let $T>0$ be finite. Suppose that $V$ is divergence-free and
\begin{equation}\label{Vassumptioninitialtime}
V\in L^{\infty}_{T}L^{2},\,\,\,\,\sup_{0<t<T} t^{\frac{1}{2}(1-\delta)}\|V(\cdot,t)\|_{L^{\infty}_{x}}<\infty.
\end{equation}
Furthermore, suppose that there exists $q\in (3,\infty)$ and $s\in (-1+\frac{2}{q},0)$ with
\begin{equation}\label{u0initialtime}
u_{0}\in J(\mathbb{R}^3)\cap \dot{B}^{s}_{q,q}.
\end{equation}
Assume that $U\in C_{w}([0,T]; J(\mathbb{R}^3))\cap L^{2}_{T} \dot{H}^{1}$ is a weak solution to the equation
\begin{equation}\label{Uequationinitialtime}
\partial_{t} U-\Delta U+V\cdot\nabla U+U\cdot\nabla V+U\cdot\nabla U+V\cdot\nabla V+ \nabla \Pi=0\,\,\,\textrm{in}\,\,\,\mathbb{R}^3\times (0,T)
\end{equation}
\begin{equation}\label{Uidinitialtime}
\textrm{div}\,U=0,\,\,\,U(\cdot,0)=u_{0}.
\end{equation}
Furthermore, assume that $U$ satisfies the energy inequality  for $t\in [0,T]$: 
\begin{equation}\label{Uenergyinequalityestnearinitialtime}
 \|U(\cdot,t)\|_{L^{2}(\mathbb{R}^3)}^2+2\int\limits_{0}^{t}\int\limits_{\mathbb{R}^3} |\nabla U(y,s)|^2 dyds\leq \|u_{0}\|_{L^{2}(\mathbb{R}^3)}^2+ 2\int\limits_{0}^{t}\int\limits_{\mathbb{R}^3} (V\otimes U+ V\otimes V): \nabla U dyds.
 \end{equation} 
Then the above assumptions imply that there exists a $\gamma(s,q,\delta)>0$ such that
\begin{equation}\label{fluctuationest}
\sup_{0<t<T} \frac{\|U(\cdot,t)-e^{t\Delta}u_{0}\|_{L^{2}_{x}}}{t^{\frac{\gamma}{2}}}<\infty.
\end{equation}
\end{lemma}
\begin{proof}
Since $V\in L^{\infty}_{t}L^{2}_{x}$ and $V$ belongs to subcritical spaces, the proof of the above can be completed using similar reasoning in the author's paper on weak-strong uniqueness ( in particular Lemma 1.5 in \cite{barker2018},which treated the case $V\equiv 0$). For the case $V\equiv 0$, we also refer to section C.3.2 in \cite{BP18}. Since the adjustments required due to $V$ are insignificant, we omit the details.
\end{proof}

\noindent{\textbf{Proof of Proposition \ref{persistregularity}}}\\
First, it is known that $U$ satisfies the mild formulation of the Navier-Stokes equations
\begin{equation}\label{Umild}
U(\cdot,t)= e^{t\Delta} u_{0}^{1}+L(V\otimes U+U\otimes V+V\otimes V)(\cdot,t)+L(U\otimes U)(\cdot,t).
\end{equation}
First, using \eqref{u0persist}, it is immediate that
\begin{equation}\label{semigroupsobolevpersist}
e^{t\Delta}u_{0}^{1}\in L^{\infty}(0,\infty; H^{\hat{\alpha}}).
\end{equation}
Using that $U\in C_{w}([0,T]; J(\mathbb{R}^3))$ and \eqref{Vassumptionpersist}, it easily follows that there exists a $p(\delta)>2$ such that
$$V\otimes U+U\otimes V+V\otimes V\in L^{p}_{T}L^{2}_{x}. $$
We can therefore apply Lemma \ref{sobolevsemigroup} to get that there exists a positive $\hat{\alpha_{1}}(p)$ such that
\begin{equation}\label{linearpersist}
L(V\otimes U+U\otimes V+V\otimes V)(\cdot,t)\in L^{\infty}(0,T; H^{\hat{\alpha_{1}}}).
\end{equation}
The term $L(U\otimes U)(\cdot,t)$ requires more work.
First, notice that by Lebesgue interpolation
$$(2^{j(\frac{\hat{\alpha}}{2}-\frac{1}{2})}\|\dot{\Delta}_{j} u^{1}_{0}\|_{L_{4}})^4\leq(2^{-j}\|\dot{\Delta}_{j}u_{0}^{1}\|_{L^{\infty}_{x}})^2(2^{j\hat{\alpha}}\|\dot{\Delta}_{j}u_{0}^{1}\|_{L^{2}_{x}})^2.  $$
Thus, $u_{0}^{1}\in \dot{B}^{\frac{\hat{\alpha}}{2}-\frac{1}{2}}_{4,4}$. Since $\hat{\alpha}\in (0,1)$, we may apply Lemma \ref{estimatenearinitialtime} to infer that there exists $\gamma(\hat{\alpha})>0$ such that
\begin{equation}\label{estinitialwithinprop}
\sup_{0<t<T} \frac{\|U(\cdot,t)-e^{t\Delta}u_{0}^{1}\|_{L^{2}_{x}}}{t^{\frac{\gamma}{2}}}<\infty.
\end{equation}
By the heat-flow characterization of homogeneous Besov spaces, we have
\begin{equation}\label{heatflowL4}
\|e^{t\Delta}u_{0}^{1}\|_{L^{4}(\mathbb{R}^3)}\leq\frac{C\|u_{0}^{1}\|_{\dot{B}^{\frac{\hat{\alpha}}{2}-\frac{1}{2}}_{4,4}}}{t^{\frac{1}{4}(1-\hat{\alpha})}}.
\end{equation}
We also get that $\sup_{0<t} t^{\frac{1}{2}}\|e^{t\Delta} u_{0}^{1}\|_{L^{\infty}(\mathbb{R}^3)}\leq C_{univ}\|u_{0}^{1}\|_{\dot{B}^{-1}_{\infty,\infty}(\mathbb{R}^3)}.$  Using this in conjunction with \eqref{Ucritical}, we infer
\begin{equation}\label{criticalfluctuation}
\sup_{0<t<T} t^{\frac{1}{2}}\|U(\cdot,t)-e^{t\Delta}u_{0}^{1}\|_{L^{\infty}(\mathbb{R}^3)}<\infty.
\end{equation}
Now, we write
\begin{align}
 \begin{split}\label{LUUexpression}
 &L(U\otimes U)(\cdot,t)= L((U-e^{t\Delta} u_{0}^{1})\otimes (U-e^{t\Delta} u_{0}^{1})+e^{t\Delta} u_{0}^{1}\otimes (U-e^{t\Delta} u_{0}^{1})+ (U-e^{t\Delta} u_{0}^{1})\otimes e^{t\Delta}u_{0}^{1}))(\cdot,t)\\&+L(e^{t\Delta}u_{0}^{1}\otimes e^{t\Delta}u_{0}^{1} )(\cdot,t).
 \end{split}
 \end{align}
 Using \eqref{estinitialwithinprop}-\eqref{criticalfluctuation}, we infer that there exists $q(\hat{\alpha}, \gamma)>2$ such that
 $$ (U-e^{t\Delta} u_{0}^{1})\otimes (U-e^{t\Delta} u_{0}^{1})+e^{t\Delta} u_{0}^{1}\otimes (U-e^{t\Delta} u_{0}^{1})+ (U-e^{t\Delta} u_{0}^{1})\otimes e^{t\Delta}u_{0}^{1}+e^{t\Delta}u_{0}^{1}\otimes e^{t\Delta}u_{0}^{1} \in L^{q}_{T}L^{2}_{x}. $$
 We then can apply Lemma \ref{lemmalinearsobolevest} to deduce that there exists $\hat{\alpha}_{2}(q)>0$ such that
 \begin{equation}\label{persistnonlinear}
 L(U\otimes U)(\cdot,t)\in L^{\infty}(0,T; H^{\hat{\alpha}_{2}(q)}).
 \end{equation}
Using this, along with \eqref{Umild}- \eqref{linearpersist}, we get the desired conclusion with $\alpha=\min(\hat{\alpha},\hat{\alpha_{1}}, \hat{\alpha_{2}})>0$.

\begin{remark}\label{Usupercritical}
Let $\gamma\in (0,\frac{1}{2})$ be as in \eqref{estinitialwithinprop}.
Note that the Proposition still holds if the critical assumption \eqref{Ucritical} is replaced by the supercritical assumption
$$\sup_{0<t<T} t^{\frac{\tilde{\gamma}}{2}}\|U(\cdot, t)\|_{L^{\infty}(\mathbb{R}^3)}<\infty\,\,\,\,\textrm{with}\,\,\,\,\tilde{\gamma}\in (1, \gamma+1).$$
\end{remark}
\end{subsection}
\end{section}
\begin{section}{Reynold's stress }
\subsection{Reynold's stress expression}
Suppose that $a$ and $b$ are Schwartz functions. In the proof of Theorem \ref{Chemin} in \cite{chemin}, Chemin evaluates  the Reynold's stress \eqref{Reynoldsstresschemin} of the form
$$B_{j}(a,b):= \dot{S}_{j}(a)\dot{S}_{j}(b)- \dot{S}_{j}(ab).$$ 
To do this, Chemin uses a decomposition of the Reynold's stress that somewhat resembles the classical paraproduct decomposition. We recall Chemin's expression now. First if $N_{0}$ is a large enough integer, it is argued in \cite{chemin} that
\begin{equation}\label{cheminhighlow}
\dot{S}_{j-N_{0}}a \dot{S}_{j-N_{0}} b= \dot{S}_{j}(\dot{S}_{j-N_{0}}a \dot{S}_{j-N_{0}} b).
\end{equation}
Indeed, $\dot{S}_{j-N_{0}} a$  and $\dot{S}_{j-N_{0}}b$ have fourier transforms  supported  in the ball $B(0,\frac{4}{3}2^{j-N_{0}})$. Thus, in Fourier space $\dot{S}_{j-N_{0}} a \dot{S}_{j-N_{0}}b$ is supported in $B(0,\frac{8}{3}2^{j-N_{0}})\subset B(0, 2^{j-N_{0}+2})$. Next the Fourier multiplier of $\dot{\Delta}_{j'}$ is compactly supported on the annulus  $\{ \xi \in \R^d :  2^{j'}\frac 3 4 \leq |\xi| \leq  2^{j'}\frac 8 3 \}$. Thus
$$ 2^{j'}\frac{3}{4}\geq 2^{j-N_{0}+2}\Rightarrow \dot{\Delta}_{j'}(\dot{S}_{j-N_{0}}(a)\dot{S}_{j-N_{0}}(b))=0.$$
If $N_{0}$ is a fixed large integer (which can be chosen to be independent of $j$, for example $N_{0}=3$), the above is satisfied for all $j'\geq j$. In particular this gives \eqref{cheminhighlow}.

Using \eqref{cheminhighlow}, in \cite{chemin}  Chemin writes
$$ B_{j}(a,b)= \dot{S}_{j}(a)\dot{S}_{j}(b)- \dot{S}_{j-N_{0}}(a)\dot{S}_{j-N_{0}}(b)+ \dot{S}_{j}(\dot{S}_{j-N_{0}}a\dot{S}_{j-N_{0}}b -ab).$$
So
$B_{j}(a,b)= \sum_{k=1}^{4} B_{j}^{k}(a,b)$
Here,
\begin{equation}\label{B1def}
B^{1}_{j}(a,b):=(\dot{S}_{j}-\dot{S}_{j-N_{0}})a \dot{S}_{j}(b)
\end{equation}
\begin{equation}\label{B2def}
B^{2}_{j}(a,b):=\dot{S}_{j-N_{0}}(a)(\dot{S}_{j}-\dot{S}_{j-N_{0}})b 
\end{equation}
\begin{equation}\label{B3def}
B^{3}_{j}(a,b):=\dot{S}_{j}((\dot{S}_{j-N_{0}}-I)a \dot{S}_{j-N_{0}}b))
\end{equation}
\begin{equation}\label{B4def}
B^{4}_{j}(a,b):=\dot{S}_{j}(a(\dot{S}_{j-N_{0}}-I)b )
\end{equation}
Furthermore, $B^{4}_{j}(a,b)= B^{3}_{j}(b,a)- B^{41}_{j}(a,b).$
Here,
\begin{equation}\label{B41def}
B^{41}_{j}(a,b):= \dot{S}_{j}((I-\dot{S}_{j-N_{0}})a(I-\dot{S}_{j-N_{0}})b)
\end{equation}
In order to reduce the frequency interactions in $B^{41}_{j}(a,b)$, Chemin uses that for a large enough fixed integer $N_{1}$
\begin{equation}\label{cheminlowhigh} 
j'\geq j+N_{1},\,\,\,\,\,j''\leq j'-2\Rightarrow \dot{S}_{j}(\dot{\Delta}_{j'} a\dot{\Delta}_{j''} b)=0.
\end{equation}
To see this, notice that
$$\textrm{supp}\, \mathcal{F}(\dot{\Delta}_{j'} a\dot{\Delta}_{j''} b)\subset 2^{j'}C+2^{j''}C $$
with $C=\{\xi: \frac 3 4\leq |\xi|\leq \frac 8 3\}.$ Now for $j''\leq j'-2$ we have $$2^{j'}C+2^{j''}C\subset 2^{j'}C+ B(0, 2^{j'-2}{8}/{3})\subset 2^{j'}C'. $$ Here, $C'=\{\xi: \frac{1}{12}\leq |\xi|\leq \frac{10}{3}\}.$
Next the Fourier multiplier of $\dot{S}_{j}$ is compactly supported on the ball $B(0,\frac{4}{3}2^{j})$.
Thus, $$ 2^{j}\frac{4}{3}< \frac{2^{j'}}{12}\Rightarrow \dot{S}_{j}(\dot{\Delta}_{j'}a\dot{\Delta}_{j''}b)=0.$$
If $j'\geq j+N_{1}$, the above is true  provided $N_{1}\geq 5$. This then gives \eqref{cheminlowhigh}.
Using \eqref{cheminhighlow}, Chemin writes
\begin{equation}\label{cheminB41}
B^{41}_{j}(a,b):= \dot{S}_{j}\Big( \sum_{j',j''> j+N_{1}, |j''-j'|< 2} \dot{\Delta}_{j'} a\dot{\Delta}_{j''} b\Big)+\dot{S}_{j}\Big( \sum_{j',j''= j-N_{0}}^{j+N_{1}} \dot{\Delta}_{j'} a\dot{\Delta}_{j''} b\Big).
\end{equation}
\begin{subsection}{Reynold's stress estimate}
Let $u_{0}\in J(\mathbb{R}^3)\cap VMO^{-1}(\mathbb{R}^3)\cap  \dot{B}^{s}_{q,q}$ with $s\in (-1+\frac{2}{q},0)$, then we apply Corollary \ref{Decomp} show that
$u_{0}=u_{0}^{1}+u_{0}^2$  with
$$u_{0}^{2} \in \dot{{B}}^{-1+\frac{3}{p}+\delta}_{p,p}(\mathbb{R}^3)\cap J(\mathbb{R}^3)\,\,\,\,\,\textrm{and}\,\,\,\,\,\,u_{0}^{1}\in\dot{H}^{\hat{\alpha}}\cap J(\mathbb{R}^3).$$
Here, $p\in (4,\infty)$, $\delta\in (0, 1-\frac{3}{p})$ and $\hat{\alpha}\in (0,\frac{3}{2}).$
From Proposition \ref{interpolativeinequalitybahourichemindanchin}, we see that $u_{0}^{2}\in \dot{B}^{-1+\frac{3}{p}}_{p,p}\subset VMO^{-1}$. Thus, $$u_{0}^{1}\in\dot{H}^{\hat{\alpha}}\cap VMO^{-1}(\mathbb{R}^3)\cap J(\mathbb{R}^3).$$
  Furthermore, by \eqref{cheminlernerheat} and the Sobolev embedding  we have 
$$e^{t\Delta}u^{2}_{0}\in \tilde{L}^{\infty}_{T}(\dot{B}^{-1+\frac{3}{p}}_{p,\infty})\cap\tilde{L}^{2}_{T}(\dot{B}^{\frac{3}{p}}_{p,\infty})\cap \tilde{L}^{1}_{T}(\dot{B}^{1}_{\infty,\infty})\cap L^{\infty}_{T} L^{2}. $$
Recall we consider $U=u-e^{t\Delta}u_{0}^{2}$ which solve the \textit{perturbed} Navier-Stokes equations.
\begin{equation}\label{perturbedNSErepeat}
\partial_{t}U-\Delta U+U\cdot\nabla  U+ e^{t\Delta} u_{0}^{2}\cdot\nabla U+  U\cdot\nabla e^{t\Delta} u_{0}^{2}+\nabla P= -e^{t\Delta} u_{0}^{2}\cdot\nabla e^{t\Delta} u_{0}^{2}
\end{equation}
\begin{equation}
\textrm{div}\,U=0,\,\,\,\,\,\,U(\cdot,0)=u_{0}^{1}
\end{equation}
Then $\dot{S}_{j}U$ satisfies the equation
\begin{equation}\label{perturbedNSElowfrequencies}
\partial_{t}\dot{S}_{j}U-\Delta \dot{S}_{j}U+\dot{S}_{j}U\cdot\nabla  \dot{S}_{j}U+ e^{t\Delta} u_{0}^{2}\cdot\nabla \dot{S}_{j}U+  \dot{S}_{j}U\cdot\nabla e^{t\Delta} u_{0}^{2}+\nabla P_{j}=\nabla\cdot F_{j} -e^{t\Delta} u_{0}^{2}\cdot\nabla e^{t\Delta} u_{0}^{2}
\end{equation}
\begin{equation}
\textrm{div}\,\dot{S}_{j}U=0,\,\,\,\,\,\,\dot{S}_{j}U(\cdot,0)=\dot{S}_{j}u_{0}^{1}.
\end{equation}
 Using that $U$ and $e^{t\Delta} u_{0}^{2}$ are divergence free, we  can write  \footnote{Here,  $a\otimes b$ denotes a matrix with $(a\otimes b)_{ij}=a_{i}b_{j}$. Furthermore, $(\nabla\cdot(a\otimes b))_{i}= \partial_{l}(a_{l}b_{i})$. Here we adopt the Einstein summation convention.}
 \begin{equation}\label{Fjdef}
 F_{j}:=\sum_{k=1}^{6} F^{(k)}_{j}
 \end{equation}
\begin{equation}\label{F1jdef}
F_{j}^{(1)}:=\dot{S}_{j}U\otimes \dot{S}_{j}U-\dot{S}_{j}(U\otimes U)
\end{equation}
\begin{equation}\label{F2jdef}
F_{j}^{(2)}:=(e^{t\Delta} u_{0}^{2}- \dot{S}_{j}(e^{t\Delta} u_{0}^{2}))\otimes \dot{S}_{j}U
\end{equation}
\begin{equation}\label{F3jdef}
F_{j}^{(3)}:=\dot{S}_{j} (e^{t\Delta}u_{0}^2)\otimes \dot{S}_{j}U-\dot{S}_{j}(e^{t\Delta}u_{0}^2\otimes U)
\end{equation}
\begin{equation}\label{F4jdef}
 F_{j}^{(4)}:=\dot{S}_{j}U\otimes (e^{t\Delta} u_{0}^{2}- \dot{S}_{j}(e^{t\Delta} u_{0}^{2}))
\end{equation}
\begin{equation}\label{F5jdef}
F_{j}^{(5)}:=\dot{S}_{j}U\otimes \dot{S}_{j}(e^{t\Delta} u_{0}^{2}))-\dot{S}_{j}(U\otimes (e^{t\Delta} u_{0}^{2}))
\end{equation}
\begin{equation}\label{F6jdef}
F_{j}^{(6)}:=e^{t\Delta} u_{0}^{2}\otimes e^{t\Delta} u_{0}^{2}-\dot{S}_{j}(e^{t\Delta} u_{0}^{2}\otimes e^{t\Delta} u_{0}^{2})).
\end{equation}
Now, we state a proposition regarding estimates of $F_{j}$, which is a crucial ingredient in proving Theorem \ref{maintheo}.
\begin{pro}\label{mainreynoldsest}
Suppose that there exists $\alpha\in(0,\frac{3}{2})$ and finite $T>0$ such that
\begin{equation}\label{assumptionreynolds1}
U\in \tilde{L}^{\infty}_{T}\dot{B}^{-1}_{\infty,\infty}\cap \tilde{L}^{1}_{T}\dot{B}^{1}_{\infty,\infty}\cap L^{\infty}_{T} H^{\alpha}.
\end{equation}
Furthermore, suppose that there exists $p>4$ and a $\delta\in (0, 1-\frac{3}{p})$ such that
\begin{equation}\label{assumptionreynolds2}
u_{0}^{2}\in\dot{B}^{-1+\frac{3}{p}+\delta}_{p,p}\cap J(\mathbb{R}^3).
\end{equation}
With this $U$ and $u_{0}^{2}$ let $F_{j}$ be defined by \eqref{Fjdef}-\eqref{F6jdef}. Then we conclude that
\begin{align}
\begin{split}\label{stressestfinal}
&\|F_{j}\|_{L^{2}_{T}L^{2}_{x}}\leq c(\alpha) 2^{-j\alpha}\|U\|_{L^{\infty}_{T} \dot{H}^{\alpha}}( j\|U\|_{\tilde{L}^{2}_{T}(\dot{B}^{0}_{\infty,\infty})}+T^{\frac{1}{2}}\|U\|_{L^{\infty}_{T}(\dot{B}^{-1}_{\infty,\infty})})\\&+
c(\alpha) 2^{-j\alpha}\|U\|_{L^{\infty}_{T} \dot{H}^{\alpha}}( j\|e^{t\Delta}u_{0}^{2}\|_{\tilde{L}^{2}_{T}(\dot{B}^{0}_{\infty,\infty})}+T^{\frac{1}{2}}\|e^{t\Delta}u_{0}^{2}\|_{L^{\infty}_{T}(\dot{B}^{-1}_{\infty,\infty})})\\&+
C(p)2^{-\frac{j}{p}}\|e^{t\Delta}u_{0}^{2}\|_{\tilde{L}^{2}_{T}(\dot{B}^{\frac{3}{p}}_{p,\infty})}\|U\|_{\tilde{L}^{\infty}_{T}(\dot{B}^{-1}_{\infty,\infty})}^{\frac{2}{p}}\|U\|_{L^{\infty}_{T}(L^{2})}^{1-\frac{2}{p}}+C(\delta)T^{\frac{\delta}{4}}2^{-j\frac{\delta}{2}}\|u^{2}_{0}\|_{L_{2}}\|u^{2}_{0}\|_{\dot{B}^{-1+\frac{3}{p}+\delta}_{p,p}}.
\end{split}
\end{align}
\end{pro}
\begin{remark}\label{reynoldsstresscompact}
In proving Theorem \ref{maintheo}, we are concerned with an  estimate for $F_{j}$ that exhibits exponential decay for large frequencies. The precise estimate above is not required, so we will often use the more compact estimate
\begin{equation}\label{reynoldsstresscompactest}
\|F_{j}\|_{L^{2}_{T}L^{2}_{x}}\leq C_{U, u_{0}, p, \alpha,T,\delta} 2^{-j\gamma_{\alpha,p,\delta}}.
\end{equation}
Here, $\gamma_{\alpha,p,\delta}:=\min(\frac{\alpha}{2}, \frac{1}{p}, \frac{\delta}{2})>0$.
\end{remark}
Proposition \ref{mainreynoldsest} will immediately follow as the result of three lemmas, the first of which handles the estimate of $F_{j}^{(1)}$.
\begin{lemma}\label{reynoldsestbmosobolev}
Suppose that for some $\alpha>0$ and some finite $T$ 
\begin{equation}\label{assumptionsobolevtheorem}
b\in \tilde{L}^{\infty}_{T}\dot{B}^{-1}_{\infty,\infty}\cap \tilde{L}^{1}_{T}\dot{B}^{1}_{\infty,\infty}\cap L^{\infty}_{T} H^{\alpha}.
\end{equation}
Then we conclude
\begin{equation}\label{bmosoboleveqn}
\|B_{j}(b,b)\|_{L^{2}_{T}L^{2}_{x}}\leq c(\alpha) 2^{-j\alpha}\|b\|_{L^{\infty}_{T} \dot{H}^{\alpha}}( j\|b\|_{\tilde{L}^{2}_{T}(\dot{B}^{0}_{\infty,\infty})}+T^{\frac{1}{2}}\|b\|_{L^{\infty}_{T}(\dot{B}^{-1}_{\infty,\infty})})
\end{equation}
\end{lemma}
\begin{remark}
Notice that if $b\in \tilde{L}^{\infty}_{T}\dot{B}^{-1}_{\infty,\infty}\cap \tilde{L}^{1}_{T}\dot{B}^{1}_{\infty,\infty}$ then 
$$\|\|\dot{\Delta}_{j} b\|_{L^{\infty}_{x}}\|_{L^{2}_{T}}\leq \|2^{j}\|\dot{\Delta}_{j} b\|_{L^{\infty}_{x}}\|_{L^{1}_{T}}^{\frac{1}{2}}\|\|2^{-j}\dot{\Delta}_{j} b\|_{L^{\infty}_{x}}\|_{L^{\infty}_{T}}^{\frac{1}{2}}. $$ Thus 
$$\|b\|_{\tilde{L}^{2}_{T}(\dot{B}^{0}_{\infty,\infty})}\leq \|b\|_{\tilde{L}^{\infty}_{T}(\dot{B}^{-1}_{\infty,\infty})}^{\frac{1}{2}}\|b\|_{\tilde{L}^{1}_{T}(\dot{B}^{-1}_{\infty,\infty}})^{\frac{1}{2}}$$
\end{remark}
\begin{proof}
First we estimate $\dot{S}_{j} b= \sum_{k\leq j-1} \dot{\Delta}_{k} b$. In particular we have,
\begin{align} 
\begin{split}\label{lowfrequenciesestsobolev}
&\|\dot{S}_{j}b\|_{L^{2}_{T}L^{\infty}_{x}}\leq \sum_{k\leq 0} \|\dot{\Delta}_{k} b\|_{L^{2}_{T}L^{\infty}_{x}}+ \sum_{k=0}^{j-1} \|\dot{\Delta}_{k} b\|_{L^{2}_{T}L^{\infty}_{x}}\\&\leq j\|b\|_{\tilde{L}^{2}_{T} (\dot{B}^{0}_{\infty,\infty})}+ T^{\frac{1}{2}}\sum_{k\leq 0} 2^{k}\|2^{-k}\|\dot{\Delta}_{k} b\|_{L^{\infty}_{x}}\|_{L^{\infty}_{T}}\\&\leq C(j\|b\|_{\tilde{L}^{2}_{T} (\dot{B}^{0}_{\infty,\infty})}+T^{\frac{1}{2}}\|b\|_{\tilde{L}^{\infty}_{T}(\dot{B}^{-1}_{\infty,\infty})}).
\end{split}
\end{align}
Now, let $N_{0}$ be a fixed integer as in the expression of the Reynold's stress in the previous subsection. Next we estimate $$\dot{S}_{j}b-\dot{S}_{j-N_{0}} b=\sum_{k=j-N_{0}}^{j-1} \dot{\Delta}_{k} b .$$
In particular,
\begin{align}
\begin{split}\label{jthfrequencyestsobolev}
&\|\dot{S}_{j}b-\dot{S}_{j-N_{0}} b\|_{L^{\infty}_{T}L^{2}} \leq \sum_{k=j-N_{0}}^{j-1} 2^{-k\alpha}(2^{k\alpha}\|\dot{\Delta}_{k} b\|_{L^{\infty}_{T}L^{2}})\\&\leq \|b\|_{L^{\infty}_{T} \dot{H}^{\alpha}} \sum_{k=j-N_{0}}^{j-1} 2^{-k\alpha}\leq C_{N_{0}, \alpha} \|b\|_{L^{\infty}_{T} \dot{H}^{\alpha}} 2^{-j\alpha}.
\end{split}
\end{align}
Next, we estimate $b-\dot{S}_{j-N_{0}} b=\sum_{k=j-N_{0}}^{\infty} \dot{\Delta}_{k} b $. We get
\begin{align}
\begin{split}\label{highfrequencyestsobolev}
&\|b-\dot{S}_{j-N_{0}} b\|_{L^{\infty}_{T}L^{2}} \leq \sum_{k=j-N_{0}}^{\infty} 2^{-k\alpha}(2^{k\alpha}\|\dot{\Delta}_{k} b\|_{L^{\infty}_{T}L^{2}})\\&\leq \|b\|_{L^{\infty}_{T} \dot{H}^{\alpha}} \sum_{k=j-N_{0}}^{\infty} 2^{-k\alpha}\leq C_{N_{0}, \alpha} \|b\|_{L^{\infty}_{T} \dot{H}^{\alpha}} 2^{-j\alpha}.
\end{split}
\end{align}
Using \eqref{lowfrequenciesestsobolev}-\eqref{highfrequencyestsobolev}, together with the continuity of $\dot{S}_{j}$ on Lebesgue spaces (with bounds independent of $j$), it is not difficult to see that the following holds. Namely for $i=1,2,3$ ($B^{i}_{j}(b,b)$ as defined in the previous subsection), we have
$$ \|B_{j}^{i}(b,b)\|_{L^{2}_{T} L^{2}}\leq c(\alpha, N_{0}) 2^{-j\alpha}\|b\|_{L^{\infty}_{T} \dot{H}^{\alpha}}( j\|b\|_{\tilde{L}^{2}_{T}(\dot{B}^{0}_{\infty,\infty})}+T^{\frac{1}{2}}\|b\|_{L^{\infty}_{T}(\dot{B}^{-1}_{\infty,\infty})}).$$
Now, we must estimate $B^{41}_{j}(b,b)$.
 Using the continuity of $\dot{S}_{j}$ on Lebesgue spaces (with bounds independent of $j$) and H\"{o}lder's inequality, we get 
 \begin{align}
 \begin{split}
 &\|B^{41}_{j}(b,b)\|_{L^{2}_{T}L^{2}_{x}}\leq \sum_{j',j''> j+N_{1}, |j''-j'|< 2} \|\dot{\Delta}_{j'} b\|_{L^{\infty}_{T}L^{2}_{x}}\|\dot{\Delta}_{j''} b\|_{L^{2}_{T}L^{\infty}_{x}}+ \sum_{j',j''= j-N_{0}}^{j+N_{1}} \|\dot{\Delta}_{j'} b\|_{L^{\infty}_{T}L^{2}_{x}}\|\dot{\Delta}_{j''} b\|_{L^{2}_{T}L^{\infty}_{x}}\\&\leq
 C_{N_{1}, N_{0}} \|b\|_{\tilde{L}^{2}_{T}(\dot{B}^{0}_{\infty,\infty})}\Big( \sum_{j'\geq j-N_{0}}\|\dot{\Delta}_{j'} b\|_{L^{\infty}_{T}L^{2}_{x}}\Big)\leq C_{N_{1}, N_{0}} \|b\|_{\tilde{L}^{2}_{T}(\dot{B}^{0}_{\infty,\infty})}\Big( \sum_{j'\geq j-N_{0}} 2^{-j'\alpha} 2^{j'\alpha}\|\dot{\Delta}_{j'} b\|_{L^{\infty}_{T}L^{2}_{x}}\Big)
 \\&\leq  C_{N_{1}, N_{0},\alpha} 2^{-j\alpha} \|b\|_{\tilde{L}^{2}_{T}(\dot{B}^{0}_{\infty,\infty})}\|b\|_{L^{\infty}_{T} \dot{H}^{\alpha}}.
 \end{split}
 \end{align}
 This estimate is of the required form. Thus the proof is completed.
\end{proof}
The following Lemma treats the estimates of $F^{(2)}_{j}-F^{(5)}_{j}$ in Proposition \ref{mainreynoldsest}. Specifically, \eqref{besovsoboleveqn} handles $F^{(3)}_{j}$ and $F^{(5)}_{j}$, whilst \eqref{besovsoboleveqnperturbed} deals with $F^{(2)}_{j}$ and $F^{(4)}_{j}$.
\begin{lemma}\label{besovsobolevassumption}
Suppose that  there exists $p\in (4,\infty)$ such that\begin{equation}\label{abesovassumption}
a\in \tilde{L}^{\infty}_{T}(\dot{B}^{-1+\frac{3}{p}}_{p,\infty})\cap\tilde{L}^{2}_{T}(\dot{B}^{\frac{3}{p}}_{p,\infty})\cap \tilde{L}^{1}_{T}(\dot{B}^{1}_{\infty,\infty})\cap L^{\infty}_{T} L^{2}.
\end{equation}
Suppose that there exists $\alpha\in(0,\frac{3}{2})$ such that
\begin{equation}\label{assumptionsobolevest2}
b\in \tilde{L}^{\infty}_{T}\dot{B}^{-1}_{\infty,\infty}\cap \tilde{L}^{1}_{T}\dot{B}^{1}_{\infty,\infty}\cap L^{\infty}_{T} H^{\alpha}.
\end{equation}
Then we conclude
\begin{align}
\begin{split}\label{besovsoboleveqn}
&\|B_{j}(a,b)\|_{L^{2}_{T}L^{2}_{x}}\leq  C(p)2^{-\frac{j}{p}}\|a\|_{\tilde{L}^{2}_{T}(\dot{B}^{\frac{3}{p}}_{p,\infty})}\|b\|_{\tilde{L}^{\infty}_{T}(\dot{B}^{-1}_{\infty,\infty})}^{\frac{2}{p}}\|b\|_{L^{\infty}_{T}(L^{2})}^{1-\frac{2}{p}}\\&+
c(\alpha) 2^{-j\alpha}\|b\|_{L^{\infty}_{T} \dot{H}^{\alpha}}( j\|a\|_{\tilde{L}^{2}_{T}(\dot{B}^{0}_{\infty,\infty})}+T^{\frac{1}{2}}\|a\|_{L^{\infty}_{T}(\dot{B}^{-1}_{\infty,\infty})}),
\end{split}
\end{align}
\begin{equation}\label{besovsoboleveqnperturbed}
\|\dot{S}_{j}b(a-\dot{S}_{j}a)\|_{L^{2}_{T}L^{2}_{x}}\leq  C(p)2^{-\frac{j}{p}}\|a\|_{\tilde{L}^{2}_{T}(\dot{B}^{\frac{3}{p}}_{p,\infty})}\|b\|_{\tilde{L}^{\infty}_{T}(\dot{B}^{-1}_{\infty,\infty})}^{\frac{2}{p}}\|b\|_{L^{\infty}_{T}(L^{2})}^{1-\frac{2}{p}}
\end{equation}
\end{lemma}
\begin{remark}
Notice that by Bernstein's inequality
$$\|a\|_{\tilde{L}^{\infty}_{T}(\dot{B}^{-1}_{\infty,\infty})}\leq C\|a\|_{\tilde{L}^{\infty}_{T}(\dot{B}^{-1+\frac{3}{p}}_{p,\infty})}.$$
Thus, by the previous remark $a\in \tilde{L}^{2}_{T}(\dot{B}^{0}_{\infty,\infty}).$
\end{remark}
\begin{proof}

The first estimate we need to prove Lemma \ref{besovsobolevassumption} is
\begin{equation}\label{lowfrequenciesbmo-1}
\|\dot{S}_{j} b\|_{L^{\infty}_{T}(L^{\frac{2p}{p-2}})}\leq C(p) 2^{\frac{2j}{p}}\|b\|_{\tilde{L}^{\infty}_{T}(\dot{B}^{-1}_{\infty,\infty})}^{\frac{2}{p}}\|b\|_{L^{\infty}_{T}L^{2}_{x}}^{1-\frac{2}{p}}.
\end{equation}
By interpolation
\begin{align}
\begin{split}\label{bmo-1frequencykest}
&\|\dot{\Delta}_{k} b\|_{L^{\infty}_{T}(L^{\frac{2p}{p-2}})}\leq C \|\dot{\Delta}_{k} b|_{L^{\infty}_{T}(L^{\infty})}^{\frac{2}{p}}\|\dot{\Delta}_{k} b\|_{L^{\infty}_{T}(L^{2})}^{1-\frac{2}{p}}\leq 2^{\frac{2k}{p}}(2^{-k}\|\dot{\Delta}_{k} b\|_{L^{\infty}_{T}(L^{\infty})})^{\frac{2}{p}}\|\dot{\Delta}_{k} b\|_{L^{\infty}_{T}(L^{2})}^{1-\frac{2}{p}}\\&
\leq 2^{\frac{2k}{p}}\|b\|_{L^{\infty}_{T}(\dot{B}^{-1}_{\infty,\infty})}^{\frac{2}{p}}\| b\|_{L^{\infty}_{T}(L^{2})}^{1-\frac{2}{p}}.
\end{split}
\end{align}
Here, we used $$\|\dot{\Delta}_{k} b\|_{L^{2}_{x}}\leq C_{univ}\|b\|_{L^{2}_{x}}.$$ Summation over $k\leq j-1$ then yields \eqref{lowfrequenciesbmo-1}.

Now we proceed with the main part of the proof of Lemma \ref{besovsobolevassumption}. We start with $B^{1}_{j}(a,b)$ as defined by \eqref{B1def}.
Using H\"{o}lder's inequality and \eqref{lowfrequenciesbmo-1} gives
\begin{equation}\label{bmo-1besovB1est1}
\|B^{1}_{j}(a,b)\|_{L^{2}_{T}L^{2}}\leq C(p) 2^{\frac{2j}{p}}\|b\|_{\tilde{L}^{\infty}_{T}(\dot{B}^{-1}_{\infty,\infty})}^{\frac{2}{p}}\|b\|_{L^{\infty}_{T}L^{2}_{x}}^{1-\frac{2}{p}} \|(\dot{S}_{j}-\dot{S}_{j-N_{0}})a\|_{L^{2}_{T}L^{p}}.
\end{equation}
Now,
$$\|(\dot{S}_{j}-\dot{S}_{j-N_{0}})a\|_{L^{2}_{T}L^{p}}\leq \sum_{k=j-N_{0}}^{j-1} \|\dot{\Delta}_{k} a\|_{L^{2}_{T}L^{p}}\leq \sum_{k=j-N_{0}}^{j-1} (2^{\frac{3k}{p}}\|\dot{\Delta}_{k} a\|_{L^{2}_{T}L^{p}})2^{-\frac{3k}{p}}\leq C(p,N_{0}) 2^{-\frac{3j}{p}}\|a\|_{\tilde{L}^{2}_{T}(\dot{B}^{\frac{3}{p}}_{p,\infty})} .$$
Combining this with \eqref{bmo-1besovB1est1} gives
\begin{equation}\label{bmo-1besovB1estmain}
\|B^{1}_{j}(a,b)\|_{L^{2}_{T}L^{2}}\leq C(p, N_{0})2^{-\frac{j}{p}}\|a\|_{\tilde{L}^{2}_{T}(\dot{B}^{\frac{3}{p}}_{p,\infty})}\|b\|_{\tilde{L}^{\infty}_{T}(\dot{B}^{-1}_{\infty,\infty})}^{\frac{2}{p}}\|b\|_{L^{\infty}_{T}L^{2}_{x}}^{1-\frac{2}{p}}.
\end{equation}

Next, verbatim reasoning to Lemma \ref{reynoldsestbmosobolev} gives
\begin{equation}\label{bmo-1besovB2estmain}
\|B^{2}_{j}(a,b)\|_{L^{2}_{T}L^{2}}\leq c(\alpha, N_{0}) 2^{-j\alpha}\|b\|_{L^{\infty}_{T} \dot{H}^{\alpha}}( j\|a\|_{\tilde{L}^{2}_{T}(\dot{B}^{0}_{\infty,\infty})}+T^{\frac{1}{2}}\|a\|_{L^{\infty}_{T}(\dot{B}^{-1}_{\infty,\infty})}).
\end{equation}

Now, we estimate $B_{j}^{3}(a,b)$ defined by \eqref{B3def}. Using that $\dot{S}_{j}$ is a bounded operator on Lebesgue spaces with bound independent on $j$, H\"{o}lder's inequality and \eqref{lowfrequenciesbmo-1}, we get that
\begin{align}\label{bmo-1besovB3est1}
\begin{split}
\|B^{3}_{j}(a,b)\|_{L^{2}_{T}L^{2}_{x}}&\leq C_{univ}\|\dot{S}_{j-N_{0}}a-a\|_{L^{2}_{T}L^{p}_{x}}\|\dot{S}_{j-N_{0}} b\|_{L^{\infty}_{T}(L^{\frac{2p}{p-2}}_{x})}\leq\\&\leq  C(N_{0},p)2^{\frac{2j}{p}} \sum_{k\geq j-N_{0}} \|\dot{\Delta}_{k} a\|_{L^{2}_{T}L^{p}_{x}}\|b\|_{\tilde{L}^{\infty}_{T}(\dot{B}^{-1}_{\infty,\infty})}^{\frac{2}{p}}\|b\|_{L^{\infty}_{T}L^{2}_{x}}^{1-\frac{2}{p}}.
\end{split}
\end{align}
Now,
$$\sum_{k\geq j-N_{0}} \|\dot{\Delta}_{k} a\|_{L^{2}_{T}L^{p}_{x}}\leq \sum_{k\geq j-N_{0}} (2^{\frac{3k}{p}}\|\dot{\Delta}_{k} a\|_{L^{2}_{T}L^{p}_{x}})2^{-\frac{3k}{p}}\leq C(p,N_{0}) 2^{-\frac{3j}{p}}\|a\|_{\tilde{L}^{2}_{T}(\dot{B}^{\frac{3}{p}}_{p,\infty})} .$$
Combining this with \eqref{bmo-1besovB3est1} gives
\begin{equation}\label{bmo-1besovB3abestmain}
\|B^{3}_{j}(a,b)\|_{L^{2}_{T}L^{2}_{x}}\leq C(N_{0},p)2^{\frac{-j}{p}}\|a\|_{\tilde{L}^{2}_{T}(\dot{B}^{\frac{3}{p}}_{p,\infty})} \|b\|_{\tilde{L}^{\infty}_{T}(\dot{B}^{-1}_{\infty,\infty})}^{\frac{2}{p}}\|b\|_{L^{\infty}_{T}L^{2}_{x}}^{1-\frac{2}{p}}.
\end{equation}

Next, we must estimate $B^{3}_{j}(b,a)$. We get by H\"{o}lder's inequality and the continuity of $\dot{S}_{j}$ that
$$ \|B^{3}_{j}(b,a)\|_{L^{2}_{T}L^{2}_{x}}\leq C_{univ}\|(\dot{S}_{j-N_{0}}b-b)\|_{L^{\infty}_{T}L^{2}_{x}} \|\dot{S}_{j-N_{0}}a\|_{L^{2}_{T}L^{\infty}_{x}}.$$
We then use \eqref{lowfrequenciesestsobolev} and \eqref{highfrequencyestsobolev} to conclude
\begin{equation}\label{bmo-1besovB3baestmain}
\|B^{3}_{j}(b,a)\|_{L^{2}_{T}L^{2}_{x}}\leq c(\alpha, N_{0}) 2^{-j\alpha}\|b\|_{L^{\infty}_{T} \dot{H}^{\alpha}}( j\|a\|_{\tilde{L}^{2}_{T}(\dot{B}^{0}_{\infty,\infty})}+T^{\frac{1}{2}}\|a\|_{L^{\infty}_{T}(\dot{B}^{-1}_{\infty,\infty})}).
\end{equation}

Next, by identical reasoning to the previous lemma we obtain
\begin{equation}\label{bmo-1besovB41est}
\|B^{41}_{j}(a,b)\|_{L^{2}_{T}L^{2}_{x}}\leq C_{N_{1}, N_{0},\alpha} 2^{-j\alpha} \|a\|_{\tilde{L}^{2}_{T}(\dot{B}^{0}_{\infty,\infty})}\|b\|_{L^{\infty}_{T} \dot{H}^{\alpha}}.
\end{equation}
Combining the above estimates gives \eqref{besovsoboleveqn}.

Finally, we mention that the proof of \eqref{besovsoboleveqnperturbed} follows from verbatim arguments as those used for \eqref{bmo-1besovB3est1}-\eqref{bmo-1besovB3abestmain}.
\end{proof}
\end{subsection}
The final lemma below estimates $F^{(6)}_{j}$ in Proposition \ref{mainreynoldsest} and thus completes the proof.
\begin{lemma}\label{highfrequencyheat}
Suppose that for some $\delta\in (0,1)$
\begin{equation}\label{u20assumptionhighfreq}
u^{2}_{0}\in \dot{B}^{-1+\delta}_{\infty,\infty}(\mathbb{R}^3)\cap L_{2}.
\end{equation}
Then for finite $T>0$
\begin{equation}\label{highfrequencyheatest}
\|e^{t\Delta}u^{2}_{0}e^{t\Delta} u^{2}_{0}- \dot{S}_{j}(e^{t\Delta}u^{2}_{0}e^{t\Delta} u^{2}_{0})\|_{L^{2}_{T}L^{2}}\leq C(\delta)T^{\frac{\delta}{4}}2^{-j\frac{\delta}{2}}\|u^{2}_{0}\|_{L_{2}}\|u^{2}_{0}\|_{\dot{B}^{-1+\delta}_{\infty,\infty}}.
\end{equation}
\end{lemma}
\begin{proof}
By the heat flow characterization  of Besov spaces
\eqref{besovequivalentnorm}, we have
\begin{equation}\label{heatu20bddest}
\|e^{t\Delta} u_{0}^{2}\|_{L_{\infty}}\leq \frac{C(\delta)\|u^{2}_{0}\|_{\dot{B}^{-1+\delta}_{\infty,\infty}}}{t^{\frac{1}{2}(1-\delta)}}.
\end{equation}
Applying \eqref{sobolevsemigroup} with $\alpha=\frac{\delta}{2}$ gives
\begin{equation}\label{heatu20sobolevest}
\|e^{t\Delta} u_{0}^{2}\|_{\dot{H}^{\frac{\delta}{2}}}\leq \frac{C(\delta)\|u^{2}_{0}\|_{L_{2}}}{t^{\frac{\delta}{4}}}. 
\end{equation}
Using Lemma \ref{besovspacesthatform analgebra} and \eqref{heatu20bddest}-\eqref{heatu20sobolevest}, we see that
\begin{equation}\label{heatu20productsobolevest}
\|e^{t\Delta} u_{0}^{2}e^{t\Delta} u_{0}^{2}\|_{\dot{H}^{\frac{\delta}{2}}}\leq C(\delta)\|e^{t\Delta} u_{0}^{2}\|_{\dot{H}^{\frac{\delta}{2}}}\|e^{t\Delta} u_{0}^{2}\|_{L_{\infty}}\leq  \frac{C(\delta)\|u^{2}_{0}\|_{\dot{B}^{-1+\delta}_{\infty,\infty}}\|u^{2}_{0}\|_{L_{2}}}{t^{\frac{1}{2}-\frac{\delta}{4}}}. 
\end{equation}
Using this, we see that
\begin{align}
\begin{split}
&\|e^{t\Delta} u_{0}^{2}e^{t\Delta} u_{0}^{2}- \dot{S}_{j}(e^{t\Delta} u_{0}^{2}e^{t\Delta} u_{0}^{2})\|_{L^{2}_{x}}\leq \sum_{k\geq j} 2^{-k\frac{\delta}{2}}(2^{k\frac{\delta}{2}}\|\dot{\Delta}_{k}(e^{t\Delta} u_{0}^{2}e^{t\Delta} u_{0}^{2})\|_{L^{2}_{x}})\\&\leq C( \delta)2^{-j\frac{\delta}{2}}\|e^{t\Delta} u_{0}^{2}e^{t\Delta} u_{0}^{2}\|_{\dot{H}^{\frac{\delta}{2}}}\leq \frac{C( \delta)2^{-j\frac{\delta}{2}}\|u^{2}_{0}\|_{\dot{B}^{-1+\delta}_{\infty,\infty}}\|u^{2}_{0}\|_{L_{2}}}{t^{\frac{1}{2}-\frac{\delta}{4}}}. 
\end{split}
\end{align}
Integrating over $(0,T)$ then gives \eqref{highfrequencyheatest}.
\end{proof}
\begin{section}{Proof of Theorem \ref{maintheo}}
\noindent\textbf{Step 1: collecting properties of the strong solution $u(\cdot,u_{0})$}\\
Recall that  there exists $q>3$ and $s\in(-1+\frac{2}{q},0)$ such that
\begin{equation}\label{initialdataassumptionrecall}
u_{0}\in J(\mathbb{R}^3)\cap VMO^{-1}(\mathbb{R}^3)\cap\dot{{B}}^{s}_{q,\infty}.
\end{equation} 
Applying Theorem 1.3 in \cite{barker2018}, Proposition \ref{regularitycriticalbmo-1}, Proposition \ref{zhangzhangregularity} and Remark \ref{VMO-1heat} we conclude that for all $\varepsilon>0$ there exists $\hat{T}(\epsilon, u_{0})>0$ and weak Leray-Hopf solution $u(\cdot,u_{0})$ (unique on $\mathbb{R}^3\times (0,\hat{T})$) with the following properties. Namely,
\begin{equation}\label{ugraduboundedrecall}
\sup_{0<s<\hat{T}}( s^{\frac{1}{2}}\|u(\cdot,s)\|_{L^{\infty}(\mathbb{R}^3)}+ s\|\nabla u(\cdot,s)\|_{L^{\infty}(\mathbb{R}^3)})<\infty,
\end{equation}
\begin{equation}\label{largestcriticalubddrecall}
u\in L^{\infty}(0,\hat{T}; \dot{B}^{-1}_{\infty,\infty})
\end{equation}
and
\begin{equation}\label{usmallnessrecall}
\|u\|_{\tilde{L}^{1}(0,\hat{T}; \dot{B}^{1}_{\infty,\infty})}<\frac{\varepsilon}{2}.
\end{equation}
\noindent{\textbf{Step 2:  splitting the initial data and  properties of the solution to the perturbed equation}}\\
 Using \eqref{initialdataassumptionrecall},  we apply Corollary \ref{Decomp} to show that there exists $p(s,q)>\max(4,q)$, $\hat{\alpha}(s,q)\in (0,\frac{3}{2})$ and $\delta(s,q)\in (0,1-\frac{3}{p})$ such that
 \begin{equation}\label{idsplittingrecall}
u_{0}=u_{0}^{1}+u_{0}^2,\,\,\,
u_{0}^{2} \in \dot{{B}}^{-1+\frac{3}{p}+\delta}_{p,p}(\mathbb{R}^3)\cap J(\mathbb{R}^3)\,\,\,\,\,\textrm{and}\,\,\,\,\,\,u_{0}^{1}\in\dot{H}^{\hat{\alpha}}\cap J(\mathbb{R}^3).
\end{equation}
From  the continuous embedding $L^{2}_{x}\hookrightarrow \dot{B}^{-\frac{3}{2}+\frac{3}{p}}_{p,p}$  and Proposition \ref{interpolativeinequalitybahourichemindanchin}  we see that $u_{0}^{2}\in \dot{B}^{-1+\frac{3}{p}}_{p,p}\hookrightarrow VMO^{-1}\hookrightarrow \dot{B}^{-1}_{\infty,\infty}$.  Hence, $$u_{0}^{1}\in\dot{H}^{\hat{\alpha}}\cap VMO^{-1}(\mathbb{R}^3)\cap J(\mathbb{R}^3). $$

Define $U:=u-e^{t\Delta}u_{0}^{2}$. Using that $u_{0}^{2}\in VMO^{-1}$, the heat-flow characterisation of homogeneous Besov spaces, Remark \ref{VMO-1heat} and the properties of $u$ in step 1, we see that there exists $T(\hat{T}, \varepsilon, u_{0}^{2})\in (0,\hat{T}]$ such that $U$ satisfies the following properties. Namely, \begin{equation}\label{perturbgraduboundedrecall}
\sup_{0<s<{T}} s^{\frac{1}{2}}\|U(\cdot,s)\|_{L^{\infty}(\mathbb{R}^3)}<\infty,
\end{equation}
\begin{equation}\label{perturblargestcriticalubddrecall}
U\in L^{\infty}(0,{T}; \dot{B}^{-1}_{\infty,\infty})
\end{equation}
and
\begin{equation}\label{perturbsmallnessrecall}
\|U\|_{\tilde{L}^{1}(0,{T}; \dot{B}^{1}_{\infty,\infty})}<{\varepsilon}.
\end{equation}

From \eqref{idsplittingrecall}, the continuous embedding and the heat flow characterization of homogeneous Besov spaces, we have
\begin{equation}\label{heatflowu02}
\sup_{t>0} t^{\frac{1}{2}(1-\delta)}\|e^{t\Delta} u_{0}^{2}\|_{L^{\infty}(\mathbb{R}^3)}\leq C\|u_{0}^{2}\|_{\dot{B}^{-1+\frac{3}{p}+\delta}_{p,p}}\,\,\textrm{and}\,\,\,\|e^{t\Delta} u_{0}^{2}\|_{L^{2}(\mathbb{R}^3)}+ 2\int\limits_{0}^{t}\int\limits_{\mathbb{R}^3}|\nabla e^{s\Delta} u_{0}^{2}|^2 dyds=\|u^{2}_{0}\|_{L^{2}(\mathbb{R}^3)}.
\end{equation}
The fact that $u(\cdot, u_{0})$ is a weak Leray-Hopf solution, together with \eqref{idsplittingrecall} and \eqref{heatflowu02}, 

allows us to infer that $U:= u-e^{t\Delta} u_{0}^{2}\in C_{w}([0,T]; J(\mathbb{R}^3))\cap L^{2}(0,T; \dot{H}^{1}(\mathbb{R}^3))$ is a solution of the following system:
\begin{equation}\label{perturbedNSErecall}
\partial_{t}U-\Delta U+U\cdot\nabla  U+ e^{t\Delta} u_{0}^{2}\cdot\nabla U+ U\cdot\nabla e^{t\Delta} u_{0}^{2}+\nabla P= -e^{t\Delta} u_{0}^{2}\cdot\nabla e^{t\Delta} u_{0}^{2}
\end{equation}
\begin{equation}
\textrm{div}\,U=0,\,\,\,\,\,\,U(\cdot,0)=u_{0}^{1}\in\dot{H}^{\hat{\alpha}}\cap VMO^{-1}(\mathbb{R}^3)\cap J(\mathbb{R}^3).
\end{equation}
Using that $u(\cdot,u_{0})$ is a weak Leray-Hopf solution, \eqref{idsplittingrecall} and \eqref{heatflowu02}, we can use known arguments\footnote{See \cite{albrittonbarker}, \cite{barkersersve} or \cite{LR1} for example.} to infer that
 \begin{equation}\label{Uenergyinequality}
 \|U(\cdot,t)\|_{L^{2}(\mathbb{R}^3)}^2+2\int\limits_{0}^{t}\int\limits_{\mathbb{R}^3} |\nabla U(y,s)|^2 dyds\leq \|u^{1}_{0}\|_{L^{2}(\mathbb{R}^3)}^2+ 2\int\limits_{0}^{t}\int\limits_{\mathbb{R}^3} (e^{s\Delta}u_{0}^{2}\otimes U+ e^{s\Delta}u_{0}^{2}\otimes e^{s\Delta} u_{0}^{2}): \nabla U dyds.
 \end{equation}
 The above properties of $U$, together with \eqref{perturbgraduboundedrecall} and \eqref{heatflowu02}, allows us to apply Proposition \ref{persistregularity} (taking $V:= e^{t\Delta}u_{0}^{2}$). Consequently, there exists $\alpha(\hat{\alpha}, \delta)\in (0,\hat{\alpha}]$ such that
 \begin{equation}\label{Usobolevmaintheo}
 U\in L^{\infty}(0,T; H^{\alpha}(\mathbb{R}^3)).
 \end{equation}
 \noindent{\textbf{Step 3:   properties of the high frequency cut-off operator acting on $U$}}\\
Using \eqref{perturblargestcriticalubddrecall}-\eqref{perturbsmallnessrecall} we see that for $j\in\mathbb{N}$

\begin{align}
\begin{split}
&\|\nabla\dot{S}_{j} U\|_{L^{1}_{T}L^{\infty}_{x}}\leq C_{univ}\sum_{j'\leq j-1} 2^{j'}\|\dot{\Delta}_{j'} U\|_{L^{1}_{T}L^{\infty}_{x}}\\&\leq C_{univ}\Big( j\|U\|_{\tilde{L}^{1}(0,T; \dot{B}^{1}_{\infty,\infty})}+\sum_{j'\leq 0} T2^{2j'}(2^{-j'}\|\dot{\Delta}_{j'} U\|_{L^{\infty}_{T}L^{\infty}_{x}}\Big)
\leq C_{univ}j\varepsilon+ C_{univ} T\|U\|_{L^{\infty}(0,T; \dot{B}^{-1}_{\infty,\infty})}.
\end{split} 
\end{align}
By redefining $\varepsilon$, $\hat{T}$ and $T$ appropriately, we have
\begin{equation}\label{Ulowfrequencyestimatemaintheo}
\|\nabla\dot{S}_{j} U\|_{L^{1}_{T}L^{\infty}_{x}}\leq j\varepsilon\log{2}+ C_{U,T}.
\end{equation}
Since $U$ belongs to the global energy class, the high frequency cut-off $\dot{S}_{j}(U)$ also belongs to $C_{w}([0,T]; J(\mathbb{R}^3))\cap L^{2}(0,T; \dot{H}^{1}(\mathbb{R}^3))$. Additionally, $\dot{S}_{j}U$ satisfies the equation in $\mathbb{R}^3\times (0,T)$
\begin{equation}\label{perturbedNSElowfrequenciesrecall}
\partial_{t}\dot{S}_{j}U-\Delta \dot{S}_{j}U+\dot{S}_{j}U\cdot\nabla  \dot{S}_{j}U+ e^{t\Delta} u_{0}^{2}\cdot\nabla \dot{S}_{j}U+  \dot{S}_{j}U\cdot\nabla e^{t\Delta} u_{0}^{2}+\nabla P_{j}=\nabla\cdot F_{j} -e^{t\Delta} u_{0}^{2}\cdot\nabla e^{t\Delta} u_{0}^{2}
\end{equation}
\begin{equation}\label{lowfreqdivfree}
\textrm{div}\,\dot{S}_{j}U=0,\,\,\,\,\,\,\dot{S}_{j}U(\cdot,0)=\dot{S}_{j}u_{0}^{1}.
\end{equation}
Here, $F_{j}$ is defined by \eqref{Fjdef}-\eqref{F6jdef}. The properties \eqref{idsplittingrecall}, \eqref{perturblargestcriticalubddrecall}- \eqref{perturbsmallnessrecall} and \eqref{Usobolevmaintheo} allow  us apply Proposition \ref{mainreynoldsest} and Remark \ref{reynoldsstresscompact}. From this one infers that \begin{equation}\label{reynoldsstresscompactestmaintheo}
\|F_{j}\|_{L^{2}_{T}L^{2}_{x}}\leq C_{U,\delta, u_{0}, p, \alpha,T} 2^{-j\gamma_{\alpha,p,\delta}}.
\end{equation}
Here, $\gamma_{\alpha,p,\delta}:=\min(\frac{\alpha}{2}, \frac{1}{p}, \frac{\delta}{2})>0$.

Finally, it is immediate that since $U$ belongs to the energy class we have $\dot{S}_{j}U \in L^{\infty}(\mathbb{R}^3\times (0,T))\cap L^{\infty}(0,T; L^{2}_{x})$. Together with \eqref{heatflowu02} and  \eqref{reynoldsstresscompactestmaintheo}, we get that 
$$\dot{S}_{j}U\otimes  \dot{S}_{j}U+ e^{t\Delta} u_{0}^{2}\otimes \dot{S}_{j}U+ \dot{S}_{j}U\otimes e^{t\Delta} u_{0}^{2}- F_{j} +e^{t\Delta} u_{0}^{2}\otimes e^{t\Delta} u_{0}^{2}\in L^{2}_{T}L^{2}_{x}. $$
This implies that 
\begin{equation}\label{continuousL2}
\dot{S}_{j}U\in C([0,T]; J(\mathbb{R}^3))
\end{equation}
 and that the following energy equality holds for $t\in [0,T]$:
\begin{align}
\begin{split}\label{lowfreqUenergyinequality}
 &\|\dot{S}_{j}(U)(\cdot,t)\|_{L^{2}(\mathbb{R}^3)}^2+2\int\limits_{0}^{t}\int\limits_{\mathbb{R}^3} |\nabla \dot{S}_{j}(U)(y,s)|^2 dyds= \|\dot{S}_{j}u^{1}_{0}\|_{L^{2}(\mathbb{R}^3)}^2\\& +2\int\limits_{0}^{t}\int\limits_{\mathbb{R}^3} (e^{s\Delta}u_{0}^{2}\otimes \dot{S}_{j}(U)+ e^{s\Delta}u_{0}^{2}\otimes e^{s\Delta} u_{0}^{2}- F_{j}): \nabla \dot{S}_{j}(U) dyds.
\end{split}
\end{align}\\
\noindent{\textbf{Step 4:  Comparing  $\dot{S}_{j}U$ with other weak Leray-Hopf solutions}}\\
Let $v(\cdot, v_{0})$ be \textit{any} weak Leray-Hopf solution to the Navier-Stokes equations with initial data $v_{0}\in J(\mathbb{R}^3)$. 
Define $v^{1}_{0}:= v_{0}- u^{2}_{0}\in J(\mathbb{R}^3)$ and $V:= v(\cdot, v_{0})- e^{t\Delta} u_{0}^{2}$. Utilizing the same reasoning applied to $U$ in step 2, we see that $V\in C_{w}([0,T]; J(\mathbb{R}^3))\cap L^{2}_{T}\dot{H}^{1} $ satisfies the following properties in $\mathbb{R}^3\times (0,T)$. Namely, 
\begin{equation}\label{Vperturbeqn}
\partial_{t}V-\Delta V+V\cdot\nabla  V+ e^{t\Delta} u_{0}^{2}\cdot\nabla V+  V\cdot\nabla e^{t\Delta} u_{0}^{2}+\nabla \Pi= -e^{t\Delta} u_{0}^{2}\cdot\nabla e^{t\Delta} u_{0}^{2},
\end{equation}
\begin{equation}
\textrm{div}\,V=0,\,\,\,\,\,\,V(\cdot,0)=v_{0}^{1}\in J(\mathbb{R}^3)
\end{equation}
and  for $t\in [0,T]$
\begin{equation}\label{Venergyinequality}
 \|V(\cdot,t)\|_{L^{2}(\mathbb{R}^3)}^2+2\int\limits_{0}^{t}\int\limits_{\mathbb{R}^3} |\nabla V(y,s)|^2 dyds\leq \|v^{1}_{0}\|_{L^{2}(\mathbb{R}^3)}^2+ 2\int\limits_{0}^{t}\int\limits_{\mathbb{R}^3} (e^{s\Delta}u_{0}^{2}\otimes V+ e^{s\Delta}u_{0}^{2}\otimes e^{s\Delta} u_{0}^{2}): \nabla V dyds.
 \end{equation}
Now notice that $v(\cdot, v_{0})- u(\cdot, u_{0})\equiv V(\cdot, v_{0}^{1})-U(\cdot, u_{0}^{1})$ and
$v_{0}-u_{0}\equiv v^{1}_{0}-u^{1}_{0}$. Therefore, to prove Theorem \ref{maintheo} it is sufficient to show that for all $\eta\in (0,1)$ there exists positive $T(\eta,u_{0}, U,s,q)$ and $C(T,U,u_{0},s,q)$  such that for all $t\in [0,T]$ we have
%that there exists $\varepsilon>0$, $\eta\in (0,1)$ and a $T>0$ such that for all $t\in (0,T)$
\begin{equation}\label{perturbL2stab}
\frac{1}{2}\|V(t)-U(t)\|_{L_{2}}^2+\int\limits_{0}^{t} \|\nabla(V-U)(t')\|_{L_{2}}^2 dt'\leq C(T,U, u_{0},s,q)\|v_{0}^{1}-u_{0}^{1}\|_{L^{2}}^{2-2\eta}.
\end{equation}
Following Chemin's idea in \cite{chemin}, we now compare $V$ with the high frequency cut-off of $U$ $\dot{S}_{j}(U)$.
Specifically, define $W_{j}:=V-\dot{S}_{j}U$. Then $W_{j}\in C_{w}([0,T]; J(\mathbb{R}^3)\cap L^{2}_{T}\dot{H}^{1}$ is a weak solution to the following equation. Namely,
\begin{equation}\label{Wjperturbeqn}
\partial_{t}W_{j}-\Delta W_{j}+W_{j}\cdot\nabla  W_{j}+ e^{t\Delta} u_{0}^{2}\cdot\nabla W_{j}+  W_{j}\cdot\nabla e^{t\Delta} u_{0}^{2}+W_{j}\cdot\nabla \dot{S}_{j}U+\dot{S}_{j}U\cdot\nabla W_{j}+\nabla \Pi_{j}= -\nabla\cdot F_{j}
\end{equation}
\begin{equation}
\textrm{div}\,W_{j}=0,\,\,\,\,\,\,W_{j}(\cdot,0)=v_{0}^{1}-\dot{S}_{j} u_{0}^{1}\in J(\mathbb{R}^3).
\end{equation}
Since $U\in L^{\infty}_{T}L^{2}_{x}$ we have that for every $k\in\mathbb{N}$ that
\begin{equation}\label{lowfreqbddSJU}
\nabla^{k} \dot{S}_{j}(U)\in L^{\infty}(\mathbb{R}^3\times (0,T)).
\end{equation}
Using $U:= u-e^{t\Delta}u^{2}_{0}$, $u_{0}^{2}\in J(\mathbb{R}^3)$ and Proposition \ref{regularitycriticalbmo-1}, we see that 
for $\lambda\in(0,T)$ and $k=0,1\ldots$, $l=0,1\ldots$:
\begin{equation}\label{SjUsmoothbmo-1}
\partial_{t}^{l}\nabla^{k} \dot{S}_{j}U\in L^{\infty}(\mathbb{R}^3\times (\lambda, T)).
\end{equation} 
From \eqref{continuousL2}, \eqref{lowfreqbddSJU}-\eqref{SjUsmoothbmo-1} and the fact that $V$ and $\dot{S}_{j}U$ satisfy global energy inequalities, standard arguments (see \cite{LR1}, \cite{barkersersve} or \cite{albrittonbarker} for example) imply the following. Namely that for $t\in [0,T]$,  $W_{j}$ satisfies the global energy inequality
\begin{align}\label{wjenergyinequality}
\begin{split}
&\frac{1}{2}\|W_{j}(\cdot,t)\|_{L^{2}(\mathbb{R}^3)}^2+\int\limits_{0}^{t}\int\limits_{\mathbb{R}^3} |\nabla W_{j}|^2 dyds\\&\leq\frac{1}{2}\|v_{0}^{1}-\dot{S}_{j} u_{0}^{1}\|_{L^{2}(\mathbb{R}^3)}^2+\int\limits_{0}^{t}\int\limits_{\mathbb{R}^3} (F_{j}+(e^{t\Delta} u_{0}^{2}\otimes W_{j})): \nabla W_{j} dyds-\int\limits_{0}^{t}\int\limits_{\mathbb{R}^3} (W_{j}\cdot \nabla \dot{S}_{j}U)\cdot W_{j}. dyds
\end{split}
\end{align}
\noindent{\textbf{Step 5:  Conclusion}}\\
Applying the H\"{o}lder and Young inequality to \eqref{wjenergyinequality} yields
\begin{align}
\begin{split}\label{wjenergyest1}
&\|W_{j}(\cdot,t)\|_{L^{2}(\mathbb{R}^3)}^2+\int\limits_{0}^{t}\int\limits_{\mathbb{R}^3} |\nabla W_{j}|^2 dyds\leq C_{univ}\|v_{0}^{1}-\dot{S}_{j} u_{0}^{1}\|_{L^{2}(\mathbb{R}^3)}^2+C_{univ}\int\limits_{0}^{t}\int\limits_{\mathbb{R}^3} |F_{j}(y,s)|^2 dyds\\&+ C_{univ} \int\limits_{0}^{t} \|W_{j}(\cdot,s)\|_{L^{2}(\mathbb{R}^3)}^2\Big(\|e^{t\Delta}u_{0}^{2}\|_{L^{\infty}(\mathbb{R}^3)}^2+ \|\nabla \dot{S}_{j} U(\cdot,s)\|_{L^{\infty}(\mathbb{R}^3)}\Big)ds.
\end{split}
\end{align}
Using \eqref{heatflowu02} and \eqref{reynoldsstresscompactestmaintheo}, we have that for $t\in (0,T)$
\begin{align}
\begin{split}\label{wjenergyest2}
&\|W_{j}(\cdot,t)\|_{L^{2}(\mathbb{R}^3)}^2+\int\limits_{0}^{t}\int\limits_{\mathbb{R}^3} |\nabla W_{j}|^2 dyds\leq C_{univ}\|v_{0}^{1}-\dot{S}_{j} u_{0}^{1}\|_{L^{2}(\mathbb{R}^3)}^2+ C_{U, \delta,u_{0}, p, \alpha,T} 2^{-2j\gamma_{\alpha,p,\delta}}\\&+ C_{univ} \int\limits_{0}^{t} \|W_{j}(\cdot,s)\|_{L^{2}(\mathbb{R}^3)}^2\Big(\frac{\|u_{0}^{2}\|_{\dot{B}^{-1+\delta}_{\infty,\infty}}^{2}}{s^{1-\delta}}+  \|\nabla \dot{S}_{j} U(\cdot,s)\|_{L^{\infty}(\mathbb{R}^3)}\Big)ds.
\end{split}
\end{align}
Applying Gronwall's lemma gives that for $t\in [0,T]$
\begin{align}
\begin{split}\label{wjgronwall}
&\|W_{j}(\cdot,t)\|_{L^{2}(\mathbb{R}^3)}^2+\int\limits_{0}^{t}\int\limits_{\mathbb{R}^3} |\nabla W_{j}|^2 dyds\leq (C_{univ}\|v_{0}^{1}-\dot{S}_{j} u_{0}^{1}\|_{L^{2}(\mathbb{R}^3)}^2\\&+ C_{U,\delta, u_{0}, p, \alpha,T} 2^{-2j\gamma_{\alpha,p,\delta}})\exp\Big(C_{univ} '\int\limits_{0}^{t} \frac{\|u_{0}^{2}\|_{\dot{B}^{-1+\delta}_{\infty,\infty}}^{2}}{s^{1-\delta}}+  \|\nabla \dot{S}_{j} U(\cdot,s)\|_{L^{\infty}(\mathbb{R}^3)}ds\Big) \\&\leq C_{U, u_{0}, p,\delta, \alpha,T}(\|v_{0}^{1}-\dot{S}_{j} u_{0}^{1}\|_{L^{2}(\mathbb{R}^3)}^2+  2^{-2j\gamma_{\alpha,p,\delta}})\exp\Big(C_{univ} ' \|\nabla \dot{S}_{j} U\|_{L^{1}_{T}L^{\infty}}\Big).
\end{split}
\end{align}
Putting $\hat{\varepsilon}:= C_{univ}'\varepsilon$ and using \eqref{Ulowfrequencyestimatemaintheo} gives
\begin{equation}\label{wjgronwallcompact}
\|W_{j}(\cdot,t)\|_{L^{2}(\mathbb{R}^3)}^2+\int\limits_{0}^{t}\int\limits_{\mathbb{R}^3} |\nabla W_{j}|^2 dyds\leq C_{U, u_{0}, p,\delta, \alpha,T}^{'}(\|v_{0}^{1}-\dot{S}_{j} u_{0}^{1}\|_{L^{2}(\mathbb{R}^3)}^2+  2^{-2j\gamma_{\alpha,p,\delta}})2^{j\hat{\varepsilon}}
\end{equation}
Notice that the same reasoning as we used to get \eqref{wjgronwallcompact} applies to the high frequencies $U-\dot{S}_{j}(U)$ of $U$. In that case one has
\begin{equation}\label{highfrequenciesUest}
\|U-\dot{S}_{j}(U)(\cdot,t)\|_{L^{2}(\mathbb{R}^3)}^2+\int\limits_{0}^{t}\int\limits_{\mathbb{R}^3} |\nabla U-\dot{S}_{j}(U)|^2 dyds\leq C_{U, u_{0}, p,\delta, \alpha,T}^{'}(\|u_{0}^{1}-\dot{S}_{j} u_{0}^{1}\|_{L^{2}(\mathbb{R}^3)}^2+  2^{-2j\gamma_{\alpha,p,\delta}})2^{j\hat{\varepsilon}}
\end{equation}
Noting that $V-U\equiv W_{j}-(U-\dot{S}_{j}(U))$ and $v^{1}_{0}- u^{1}_{0}\equiv v^{1}_{0}-\dot{S}_{j}u_{0}^{1}+(\dot{S}_{j}u_{0}^{1}-u_{0}^{1}) $, we can combine  \eqref{wjgronwallcompact}-\eqref{highfrequenciesUest} to get
\begin{align}
\begin{split}\label{VminusUest1}
&\|(V-U)(\cdot,t)\|_{L^{2}(\mathbb{R}^3)}^2+\int\limits_{0}^{t}\int\limits_{\mathbb{R}^3} |\nabla (V-U)|^2 dyds\\ &\leq C_{U, u_{0}, p,\delta, \alpha,T}^{''}(\|u_{0}^{1}-\dot{S}_{j} u_{0}^{1}\|_{L^{2}(\mathbb{R}^3)}^2+\|u_{0}^{1}-v_{0}^{1}\|_{L^{2}_{x}}^{2}+  2^{-2j\gamma_{\alpha,p,\delta}})2^{j\hat{\varepsilon}}.
\end{split}
\end{align}
Recall from \eqref{idsplittingrecall} that $u_{0}^{1}\in \dot{H}^{\hat{\alpha}}$ with $\hat{\alpha}\in (0,\frac{3}{2})$. Thus,
$$\|u_{0}^{1}-\dot{S}_{j} u_{0}^{1}\|_{L^{2}(\mathbb{R}^3)}\leq \sum_{j'\geq j}(2^{j'\hat{\alpha}}\|\dot{\Delta}_{j'}u_{0}^{1}\|_{L^{2}_{x}}) 2^{-j'\hat{\alpha}}\leq \|u_{0}^{1}\|_{\dot{H}^{\hat{\alpha}}} \sum_{j'\geq j} 2^{-j'\hat{\alpha}}\leq C(\hat{\alpha})\|u_{0}^{1}\|_{\dot{H}^{\hat{\alpha}}} 2^{-j\hat{\alpha}}.$$
\end{section}
Combining this with \eqref{VminusUest1} gives that for all $j\in\mathbb{N}$ we have
\begin{equation}\label{VminusUmainest}
\|(V-U)(\cdot,t)\|_{L^{2}(\mathbb{R}^3)}^2+\int\limits_{0}^{t}\int\limits_{\mathbb{R}^3} |\nabla (V-U)|^2 dyds\leq C_{U, u_{0}, p,\delta, \alpha,T}^{'''}(\|u_{0}^{1}-v_{0}^{1}\|_{L^{2}(\mathbb{R}^3)}^{2}+  2^{-2j\min{(\gamma_{\alpha,p,\delta}, \hat{\alpha}})})2^{j\hat{\varepsilon}}.
\end{equation}
 For any fixed $\eta\in (0,1)$, we take $\hat{\varepsilon}$ to satisfy
\begin{equation}\label{epsilonsmallness}
\hat{\varepsilon}=2\eta\min{(\gamma_{\alpha,p,\delta}, \hat{\alpha}})).
\end{equation}
Now, we treat two cases
\begin{enumerate}
\item $u_{0}^{1}=v_{0}^{1}$
\item $0<\|u_{0}^{1}-v_{0}^{1}\|_{L^{2}_{x}}<1$.
\end{enumerate}
In the first case, we have
\begin{equation}\label{VminusUest1zeroinitialdata}
\|(V-U)(\cdot,t)\|_{L^{2}(\mathbb{R}^3)}^2+\int\limits_{0}^{t}\int\limits_{\mathbb{R}^3} |\nabla (V-U)|^2 dyds\leq C_{U, u_{0}, p,\delta, \alpha,T}^{'''}  2^{-2j\min{(\gamma_{\alpha,p,\delta}, \hat{\alpha}})+j\hat{\varepsilon}}.
\end{equation}
With the smallness assumption \eqref{epsilonsmallness}, we see that taking $j\uparrow \infty$ results in $V\equiv U$ in $\mathbb{R}^3\times (0,T)$. This recovers the author's weak-strong uniqueness result in \cite{barker2018}.

In the second case, take
\begin{equation}\label{highfrequencychoice}
j:= \ceil*{\frac{-\log_{2}(\|u_{0}^{1}-v_{0}^{1}\|_{L^{2}_{x}})}{\min{(\gamma_{\alpha,p,\delta}, \hat{\alpha}})}}>0.
\end{equation}
We then get that 
\begin{equation}\label{consequence1highfreq}
2^{-2j\min{(\gamma_{\alpha,p,\delta}, \hat{\alpha}}})\leq \|u^{1}_{0}-v^{1}_{0}\|_{L^{2}_{x}}^2
\end{equation}
and 
\begin{equation}\label{consequence2highfreq}
2^{j\hat{\varepsilon}}\leq \frac{2^{\hat{\varepsilon}}}{\|u^{1}_{0}-v^{1}_{0}\|_{L^{2}_{x}}^{2\eta}}.
\end{equation}
Substituting \eqref{consequence1highfreq}-\eqref{consequence2highfreq} into \eqref{VminusUmainest} gives
\begin{equation}\label{VminusUfinalest}
\|(V-U)(\cdot,t)\|_{L^{2}(\mathbb{R}^3)}^2+\int\limits_{0}^{t}\int\limits_{\mathbb{R}^3} |\nabla (V-U)|^2 dyds\leq C_{U, u_{0}, p,\delta, \alpha,T,\hat{\varepsilon}}\|u_{0}^{1}-v_{0}^{1}\|_{L^{2}_{x}}^{2(1-\eta)}
\end{equation}
Putting cases 1 and 2 together gives \eqref{perturbL2stab}. As explained in Step 4, this  implies the conclusion of Theorem \ref{maintheo}.
\
\end{section}
\subsection*{Acknowledgement}
I would like to thank Isabelle Gallagher for encouraging me to write this paper.
\bibliography{refs}        %use a bibtex bibliography file refs.bib
\bibliographystyle{plain}  %use the plain bibliography style

\end{document}